\documentclass[11pt,a4paper, envcountsame]{amsart}
\usepackage[usenames,dvipsnames]{color}

 \usepackage{amsmath, amssymb, xspace}
 \usepackage{graphicx}

 \usepackage[all]{xy}

\usepackage{tikz}
\usetikzlibrary{arrows,snakes,positioning,backgrounds,shadows}

\usepackage{undertilde}

\newtheorem{theorem}{Theorem}[section]

\newtheorem{thm}[theorem]{Theorem}

\newtheorem{fact}[theorem]{Fact}

\newtheorem{example}[theorem]{Example}

\newtheorem{proposition}[theorem]{Proposition}

\newtheorem{remark}[theorem]{Remark}

\newtheorem{prop}[theorem]{Proposition}

\newtheorem{claim}[theorem]{Claim}

\newtheorem{definition}[theorem]{Definition}

\newtheorem{conjecture}[theorem]{Conjecture}

\newtheorem{lemma}[theorem]{Lemma}

\newtheorem{corollary}[theorem]{Corollary}

\newtheorem{cor}[theorem]{Corollary}

\newtheorem{question}[theorem]{Question}

\newcommand{\Cyl}[1]{\ensuremath{[\![{#1}]\!]}}

\newcommand{\DII}{\Delta^0_2}
\newcommand{\NN}{{\mathbb{N}}}
\newcommand{\RR}{{\mathbb{R}}}
\newcommand{\QQ}{{\mathbb{Q}}}
\newcommand{\ZZ}{{\mathbb{Z}}}
\newcommand{\sub}{\subseteq}
\newcommand{\sN}[1]{_{#1\in \NN}}
\newcommand{\uhr}[1]{\! \upharpoonright_{#1}}
\newcommand{\ML}{Martin-L{\"o}f}
\newcommand{\SI}[1]{\Sigma^0_{#1}}
\newcommand{\PI}[1]{\Pi^0_{#1}}
\newcommand{\PPI}{\PI{1}}
\newcommand{\PP}{\mathcal P}

\newcommand{\bi}{\begin{itemize}}
\newcommand{\ei}{\end{itemize}}
\newcommand{\bc}{\begin{center}}
\newcommand{\ec}{\end{center}}

\newcommand{\Halt}{{\ES'}}
\newcommand{\ES}{\emptyset}
\newcommand{\estring}{\emptyset}
\newcommand{\ria}{\rightarrow}
\newcommand{\tp}[1]{2^{#1}}
\newcommand{\ex}{\exists}
\newcommand{\fa}{\forall}
\newcommand{\lep}{\le^+}
\newcommand{\gep}{\ge^+}
\newcommand{\la}{\langle}
\newcommand{\ra}{\rangle}
\newcommand{\Kuc}{Ku{\v c}era}
\newcommand{\seqcantor}{2^{ \NN}}
\newcommand{\cantor}{\seqcantor}
\newcommand{\strcantor}{2^{ < \omega}}
\newcommand{\strbaire}{\NN^{ < \omega}}
\newcommand{\fao}[1]{\forall #1 \, }
\newcommand{\exo}[1]{\exists #1 \, }
\newcommand{\Opcl}[1]{[#1]^\prec}
\newcommand{\leT}{\le_{\mathrm{T}}}

\newcommand{\MLR}{\mbox{\rm \textsf{MLR}}}

\newcommand{\Om}{\Omega}
\newcommand{\n}{\noindent}

\newcommand{\vsps}{\vspace{3pt}}
\newcommand{\verif}{\n {\it Verification.\ }}
\newcommand{\vsp}{\vspace{6pt}}
\newcommand{\leb}{\mathbf{\lambda}}
\newcommand{\lwtt}{\le_{\mathrm{wtt}}}
\renewcommand{\SS}{\mathcal{S}}
\newcommand{\sss}{\sigma}
\newcommand{\aaa}{\alpha}

\newcommand{\w}{\omega}

\newcommand{\lland}{\, \land \, }

\newcommand{\seq}[1]{(#1_i)_{i\in\NN}}

\newcommand\+[1]{\mathcal{#1}}

\newcommand{\wt}{\widetilde}
\newcommand{\ol}{\overline}
\newcommand{\ul}{\underline}
\newcommand{\ape}{\hat{\ }}

\newcommand{\LLR}{\Longleftrightarrow}
\newcommand{\lra}{\leftrightarrow}
\newcommand{\LR}{\Leftrightarrow}
\newcommand{\RA}{\Rightarrow}
\newcommand{\LA}{\Leftarrow}

\newcommand{\CCC}{\mathcal{C}}
\newcommand{\UM}{\mathbb{U}}
\newcommand{\sssl}{\ensuremath{|\sigma|}}

\newcommand{\dom}{\ensuremath{\mathrm{dom}}}

\def\uh{\upharpoonright}

\newcommand \DemBLR{\textup{Demuth}_{\textup{BLR}}}

\newcommand{\vectornorm}[1]{\left|\left|#1\right|\right|}

\begin{document}

\begin{abstract}
This year's logic blog has focussed on: 

1. Demuth randomness

2. traceability

3. The connection of computable analysis and randomness

4. $K$-triviality in metric spaces.

\end{abstract}

\title{Logic Blog 2011}

%\author{Ask Andr\'e Nies how add to  this}

\maketitle

%\n  A. Taveneaux, \emph{Randomness Zoo}, Logic Blog, Section 6, available at 
%
%\texttt{http://dl.dropbox.com/u/370127/Blog/Blog2011.pdf}.

\medskip

\tableofcontents

\part{Randomness, Kolmogorov complexity,  and computability}
%%%%%%%%%%%%%%%%%%%%%%%%%%%%%%%%%%%
%\input{sections2011/Jan2011Nies}

\section{Jan-May 2011:  Demuth randomness, weak Demuth randomness,   and the corresponding  lowness notions}

This section contains news on Demuth randomness and weak Demuth randomness due to Nies and \Kuc. They characterize arithmetical complexity and study the diamond class of no-weakly Demuth random.

 The main result of this section  is on    lowness for (weak) Demuth randomness, due to
Bienvenu, Downey, Greenberg, Nies, Turetsky. They characterize lowness for Demuth by a concept called BLR-traceability, which implies being in  computably dominated $\cap$  jump traceable (they  can now show that it is a proper subclass). Paper submitted Feb.\ 2012.

\subsection{Basic definitions}

\label{ss:DemuthRdnessDefs}

\n For  a set $W \sub \strcantor$, we let
\bc $\Opcl W = \{Z \in \cantor \colon \, \ex n  \, Z \uhr  n \in W \}$, \ec    the corresponding open class in Cantor space.

 \begin{definition} \label{df:DemDef} A \emph{Demuth test}  is a sequence of c.e.\ open sets $(S_m)\sN{m}$ such that $\fao m \leb S_m \le \tp{-m}$, and there is a function  $f\lwtt \Halt$ such that  $S_m = \Opcl{W_{f(m)}}$.

  A set~$Z$ \emph{passes} the test if $  Z\not \in   S_m$ for almost every~$m$.
    We say that~$Z$ is  \emph{Demuth random}  if~$Z$ passes  each Demuth test. \end{definition}

%
%\begin{definition} \label{df:wD}  In the context of Definition~\ref{df:DemDef},  if we also have $S_m \supseteq S_{m+1}$ for each $m$, we say that  $(S_m)\sN m $ is a \emph{monotonic Demuth test}. In this case the passing condition is equivalent to $Z \not \in \bigcap_m S_m$. If $Z$ passes all monotonic Demuth tests we say that $Z$ is \emph{weakly Demuth random}.
%\end{definition}

If we apply the usual passing  condition for tests, we obtain the following notion.

\begin{definition} We say that $Z$ is \emph{weakly Demuth random} if for each Demuth test $(S_m)\sN{m}$  there is an $m$ such that $Z \not \in S_m$.
\end{definition}

 \begin{remark} \label{rem:versions}  {\rm Recall that  $f\lwtt \Halt$ if and only if $f$ is $\omega$-c.e., namely, $f(x) = \lim_t g(x,t)$  for some computable function~$g$ such that the number of changes $g(x,t)\neq g(x,t-1)$ is computably bounded in~$x$.  Hence the intuition is that we can change the $m$-component $S_m$ a computably bounded number of times. We will   denote by $S_m [t]$ the version of the component $S_m$ that we have at stage $t$. Thus $S_m[t]=  \Opcl{W_{g(m,t)}}$ where~$g$ is understood to be a computable  approximation of $f$ as above. } \end{remark}

\subsection{Arithmetical Complexity}
  \Kuc{}  and Nies  worked in Prague in May. They    proved the following.

\begin{prop}\label{prop:KuNies-arithmetic-complexity-of-Demuth} (i) Weak Demuth randomness is $\PI 3$.

\n (ii) Demuth randomness is not $\PI 3$. \end{prop}

In \cite{Kucera.Nies:11} they had shown that both notions are $\PI 4$.

\begin{proof} (i)  Let $(\Gamma_e, h_e) \sN e$ be an effective listing of pairs of Turing functionals and use bounds (pc functions), such that if $p= \Gamma_e^X(m) \downarrow$ then    $\leb \Opcl {W_p} \le \tp{-m}$. If $\Gamma_e^{\Halt}$ is total with use bound $h_e$ (also total) then this determines a   Demuth test according to Definition~\ref{df:DemDef}. Conversely, we obtain all   Demuth tests that way. Then $Z$ is not weakly Demuth random iff

\bc $\ex e \, \fa m \fa s \, \ex t> s \, [ p = \Gamma_e^{\Halt}(m)[t] \downarrow  \, \text{with use bound} \, h_{e,t}(m) \lland Z \in \Opcl {W_{p,t}}]$. \ec

For,  if the pair  $(\Gamma_e^{\Halt}, h_e)$ does not determine a   Demuth test then the condition for $e$ fails.
This shows (i).

\n (ii)  Otherwise, the class of non-Demuth random ML-random sets is a $\SI 3 $ null class, so there must be a c.e.\ noncomputable set $A$ below all of them. This contradicts the fact that there is a high ML-random $\DII$ set $Z  \not \ge_T A$; this set is not GL$_1$ and hence not Demuth random. (See \cite[Thm 3.2]{Kucera.Nies:11} for a somewhat stronger result that even  implies the existence of a weakly Demuth random set $Z$.)

\end{proof}

\subsection{The diamond class of the non-weakly Demuth randoms}

  \begin{thm}[\Kuc{}  and Nies]
Let $A$ be    a c.e.\ set. Then

  \n  $A$ is Turing below all    ML-random but non-weak Demuth random  sets $\LR$

  \hfill $A$ is strongly jump traceable.
  \end{thm}

\begin{proof} $(\RA)$ An $\omega$-c.e.\ set is not weakly Demuth random. If $A$ is below all $\omega$-c.e.\ ML-randoms then $A$ is already s.j.t.\ by a result of  \cite{Greenberg.Hirschfeldt.ea:12}.

\n $(\LA)$ Suppose $(S_m) \sN m$ is a Demuth test with $S_m \supseteq S_{m+1}$, . We need to show that $A \leT Y$   for any   ML-random set $Y$ such that $Y  \in \bigcap_m S_m$.

As in Remark~\ref{rem:versions} we have  $S_m[t]=  \Opcl{W_{g(m,t)}}$ where~$g$ is a computable binary function with a computably bounded number of changes $g(m, t) \neq g(m,t-1)$.  We may assume that $S_m[t] \supseteq S_{m+1}[t]$ for each $m, t$.

For $m<s$ let
$r(m,s)  $ be the maximum $e$ such that $g(i,t) $ is stable for stages $t$, $m< t \le s$, for  all $i \le e$.
Let

\[ G_m = \bigcup_{s> m} \Opcl{ W_{g(r(m,s),s),s}} \]
Informally, at stage $s $, the $\SI 1$ set $G_m $ copies the   version $S_e$  for $e$ largest that hasn't  changed since $m$. If it does change $G_m$ takes the  (larger)  version $S_{e-1}$ etc. till this stops.

Clearly $\bigcap_e S_e$ is contained in the $\PI 2$ null class  $\bigcap_m G_m$.   Now consider the usual Hirschfeldt-Miller cost function

\[ c(m,s) = \leb G_{m,s}.\]
Recall that if a set $A$ obeys $c$ then $A $ is below each ML-random set $Y \in \bigcap_m G_m$ by \cite[Thm. 5.3.x] {Nies:book}.

This cost function $c$ is benign. For let $c(x,s) > \tp{-e}$. Then $r(x,s)< e$, so $g(i,t)$ has  changed for some $i \le e$ between stages $x$ and $s$. So benignity follows from the  hypothesis that $(S_e)$ is a Demuth test.

Now, by \cite{Greenberg.Nies:11}, the  c.e.\ strongly jump traceable set  $A$ obeys $c$.
\end{proof}

Let $\CCC$ be the class of non-weakly Demuth random sets. Then the theorem says that  $\CCC^\Diamond$ coincides with  the strongly jump traceable c.e.\ sets.
  We also considered the question whether there is a single weakly Demuth random above all c.e. strongly jump traceables. If not, then this  would mean that  the non-weakly Demuth random sets  form    the largest class   $\CCC$ with $\CCC^\Diamond $  = the c.e.\  strongly jump traceables, where for a class $\CCC$ only the intersection with $\MLR$ counts.

\subsection{Lowness for weak Demuth randomness}
%%%%%%%%%%%%%%%%%%%%%%%
%%%%%%%%%%%%%%%%%%%%%%%
For details on the following summary, see our  forthcoming paper~\cite{Bienvenu.Downey.ea:nd}.
$Z$ is weakly  Demuth \emph{by} an oracle set $A$ if it passes every weak $A$-Demuth test where the number of version changes is computably bounded.
\begin{prop}\label{prop:weakDemuthBLR}
The following are equivalent for a set~$A$:
\begin{enumerate}
\item[(i)] Every weak Demuth random is weak Demuth random by~$A$.
\item[(ii)] $A$ is $K$-trivial.
\end{enumerate}
\end{prop}

\begin{proof}
The implication (i) $\Rightarrow$ (ii) follows from a result of Bienvenu and Miller~\cite{Bienvenu.Miller:nd}.

Other direction by Claim 5.7 of Nies'  paper ``c.e.\ sets below random sets''~\cite{Nies:11}.
\end{proof}

\begin{theorem}
The only sets that are low for weak Demuth randomness are the computable sets.
\end{theorem}

Since each $K$-trivial is $\Delta^0_2$,  by   Proposition~\ref{prop:weakDemuthBLR}  it suffices to prove the following.

 \begin{lemma}
Let~$A$ be a set computing a function~$f$ which is not dominated by any computable function.  Then there is a weakly Demuth random~$Z$ which is not weakly Demuth random relative to~$A$.
\end{lemma}

\subsection{BLR-traceability}
\begin{definition} A BLR trace is a sequence $(T_n)\sN n  $ such that $T_n= W_{r(n)}$ where $r$ is an $\omega$-c.e.\ function. Let $h$ be an order function. We say $h$ is a bound for $(T_n)$ if $|T_n| \le h(n) $ for each $n$. \end{definition}

\begin{definition} We say that $A$ is BLR traceable if there is an order function $h$ such that each $\omega$-c.e.\ by $A$ function  $f$ has a BLR trace with bound $h$. \end{definition}

As usual, the choice of $h$ doesn't matter. Also, if $A$ is computably dominated then we can take equivalently $T_n= D_{p(n)}$ for some $\omega$-c.e.\ function~$p$.

\begin{fact} Jump traceable \& superlow is equivalent to BLR traceable with bound 1. \end{fact}

\begin{fact} Jump traceable \& c.e.\ implies BLR traceable with bound 1. \end{fact}
These are both by Cole/Simpson, and because having a BLR trace with bound 1 means being $\omega$-c.e.

\begin{prop}    BLR traceable with constant bound $>1$ does not  imply $\w$-c.e. \end{prop}

\begin{fact} BLR traceable implies jump traceble. \end{fact}

This is so because the function $f(x) = J^A(x)$ if defined, $0$ otherwise is $\omega$-c.e.\ by $A$.

%This might  also work for weak Demuth randomness, but have to check the clopenification process here.  NOPE

\begin{theorem}
There is an $\w$-c.e.\ set which is jump traceable but not BLR traceable.
\end{theorem}

\begin{thm}
There  is jump traceable and computably dominated set  that is not BLR traceable.
\end{thm}
Proofs are  in the paper.

%%%%%%%%%%%%%%%%%%%%%%%
\subsection{Existence}
%%%%%%%%%%%%%%%%%%%%%%%

\begin{thm}\label{thm:SpecialTreeTraceables} There is a special $\PPI$ class of BLR traceable sets. \end{thm}

\begin{proof}
Let $\{\langle \Gamma_e, g_e\rangle\}_{e \in \w}$ be an enumeration of (partial) $\w$-c.e.\ by oracle functions.  So the following hold:
\begin{itemize}
\item $g_e$ is a partial computable function which converges on an initial segment of~$\w$;
\item $\Gamma_e^X$ is total for every oracle~$X$;
%\item $\Gamma_e^X(n,0) = 0$ for every oracle~$X$;
\item $|\{ t \mid \Gamma_e^X(n,t) \neq \Gamma_e^X(n,t+1)\}| < g_e(n)$ for every~$n$ such that~$g_e(n)\downarrow$ and every oracle~$X$.
\end{itemize}
We let $f_e^X(n) = \lim_t \Gamma_e^X(n,t)$.

We build a $\Pi^0_1$-class~$\PP$ and computable BLR traces $\{({T_n^e})_{n \in \w}\})_{e \in \w}$ with bound~$2^n$.  For all~$i$, we have the requirement:
\begin{description}
\item[$R_i$] $\phi_i$ is not a computable description of a real in~$\PP$.
\end{description}
For every pair~$(e, n)$ with $e < n$, we have the requirement:
\begin{description}
\item[$Q_{e,n}$] If $g_e(n)\downarrow$, $f_e^X(n) \in T_n^e$ for all reals $X \in \PP$.
\end{description}
Clearly these requirements will suffice to prove the result.

Each strategy will receive from the previous strategy some finite collection of strings~$\{\alpha_j\}$ all of the same length, and it will create some finite extensions~$\{\beta_k\}$ (all of the same length) such that for every~$j$ there is at least one~$k$ such that~$\alpha_j \subseteq \beta_k$, and all reals in~$\bigcup_k [\beta_k]$ meet the strategy's requirement.

$R_i$-strategies will initially define exactly two~$\beta_k$ for every~$\alpha_j$, but may later remove one.

$Q_{e,n}$-strategies will define exactly one~$\beta_k$ for every~$\alpha_j$.  They may need to redefine~$\beta_k$ some finite number of times, but each new definition will be an extension of the previous.

At the end of every stage~$s$, we let~$\{\hat\beta_k\}$ be the outputs of the last strategy to act at stage~$s$, and define the tree~$P_s$ to be all strings comparable with one of the~$\hat\beta_k$.  $\PP$ will be $\bigcap_s [P_s]$.

\medskip

{\em Description of $R_i$-strategy:}

This is the standard non-computability requirement on a tree:

\begin{enumerate}
\item Let $\{\alpha_j\}_{j < m}$ be the output of the previous strategy.  For every~$j < m$, let $\beta_{j,0} = \alpha_j  0$ and $\beta_{j,1} = \alpha_j  1$.
\item Wait for~$\phi_i(|\alpha_0|)$ to converge; while waiting, let the outputs be $\{\beta_{j,0}, \beta_{j,1}\}_{j < m}$.
\item When $\phi_i(|\alpha_0|)$ converges\dots
\begin{itemize}
\item \dots if $\phi_i(|\alpha_0|) = 0$, let the outputs be $\{\beta_{j,1}\}_{j < m}$.
\item \dots if $\phi_i(|\alpha_0|) \neq 0$, let the outputs be $\{\beta_{j,0}\}_{j < m}$.
\end{itemize}
\end{enumerate}

\medskip

{\em Description of $Q_{e,n}$-strategy:}

Until~$g_e(n)$ converges, this strategy takes no action.  We ignore for the moment the computable bound on the number of times the index of~$T_e^n$ changes.

Let $\{\alpha_j\}_{j < m}$ be the output of the previous strategy.  We will keep several values to assist the strategy: $\ell_s$ will be the number of times the output has been redefined by stage~$s$; $c_s(j)$ will be the current guess for~$f_e(n)$ on an extension of~$\alpha_j$.  We initially have $\ell_s = 0$ and $c_s(j) = -1$ for all~$j$.  Unless otherwise defined, $\ell_{s+1} = \ell_s$ and $c_{s+1}(j) = c_s(j)$.

For every~$j$, let $\beta_j^0 = \alpha_j$.  We initially let the outputs be $\{\beta_j^0\}_{j \in \w}$ and define~$T_n^e = \emptyset$.  We run the following strategy, where~$s$ is the current stage:
\begin{enumerate}
\item Wait for a string~$\gamma \in P_s$ with~$\gamma$ extending one of the~$\beta_j^{\ell_s}$ and~$\Gamma_e^\gamma(n,s) \neq c_s(j)$.
\item When such a string is found for~$\beta_j^{\ell_s}$:
\begin{enumerate}
\item Define $\beta_j^{\ell_s+1} = \gamma$ and $c_{s+1}(j) = \Gamma_e^\gamma(n,s)$.
\item For every $k < m$ with $k \neq j$, choose $\beta_k^{\ell_s+1}$ extending $\beta_k^{\ell_s}$ of the same length as~$\beta_j^{\ell_s+1}$.
\item Redefine $T_n^e = \{ c_{s+1}(k) \mid k < m\}$.
\item Define $\ell_{s+1} = \ell_s + 1$.
\end{enumerate}
\item Return to Step 1.
\end{enumerate}

\medskip

{\em Full construction:}

We make the assumption that for every~$s$, there is precisely one pair~$(e,n)$ with~$e < n$ and~$g_{e,s+1}(n)\downarrow$ but~$g_{e,s}(n) \uparrow$.
%Further, if $g_{e,s}(n+1)\downarrow$, then $g_{e,s}(n)\downarrow$.
We give the $Q_{e,n}$-strategies priority based on the order in which the~$g_e(n)$ converge: if $g_e(n)$ converges before $g_{e'}(n')$, then the $Q_{e,n}$-strategy has stronger priority than the $Q_{e',n'}$-strategy.  If~$g_e(n)$ never converges, then~$Q_{e,n}$ never has a priority, but this is fine because it never acts.

We prioritize the $R_i$-strategies based on the priorities of the $Q_{e,n}$-strategies: $R_i$ has weaker priority than~$R_{i'}$ for any $i' < i$, and also weaker priority than any $Q_{e,n}$-strategy with~$n \leq i$.  It is given the strongest priority consistent with these restrictions.

Since we only consider $n > e \ge 0$, the $R_0$-strategy will always have strongest priority.  It receives $\alpha_0 = \seq{}$ as the ``output of the previous strategy''.

At stage~$s$, let~$(e,n)$ be the pair such that $g_e(n)$ has newly converged.  We initialise~$R_n$ and all strategies which had weaker priority than~$R_n$.  The priorities of the various~$R_i$ are then redetermined.  We then let all $Q$-strategies with priorities and all $R_i$-strategies with $i < s$ act, in order of priority.

Whenever a strategy redefines its output, all weaker priority strategies are initialised.

\end{proof}

\subsection{BLR Traceability is Equivalent to Lowness for Demuth${}_{\text{BLR}}$}
%%%%%%%%%%%%%%%%%%%%%%%
%%%%%%%%%%%%%%%%%%%%%%%

We say that a set $Z$ is Demuth random \emph{by}   an  oracle set $A$ ($\DemBLR$  $A$ for short) if it passes every   $A$-Demuth test where the number of version changes is computably bounded.  $A$ is low for $\DemBLR$ if every Demuth random is already $\DemBLR$ $A$.

\begin{theorem}\label{thm:DemuthBLREquivalence} The following are equivalent for a set~$A$.
\begin{enumerate}
\item[(i)] $A$ is BLR traceable.
\item[(ii)] $A$ is low for $\DemBLR$.
\end{enumerate}
\end{theorem}

\begin{corollary}
There exists a perfect class of sets which are all low for Demuth randomness.
\end{corollary}

\begin{proof}
First observe that any computably dominated set which is low for $\DemBLR$ is already low for Demuth randomness.  Now by a result of Martin/Miller (1968), the $\Pi^0_1$-class obtained in Theorem \ref{thm:SpecialTreeTraceables} contains a perfect subclass of sets which are computably dominated.  By the previous theorem, this class is as required.
\end{proof}

 \section{ The collection of weakly-Demuth-random reals is not $\mathbf{\Sigma}^0_3$}
 
 By Yu Liang.

This is a result corresponding to Proposition \ref{prop:KuNies-arithmetic-complexity-of-Demuth}.

\begin{lemma}\label{Lemma: Yu-meet-dense-by-weakly-Demuth-test}
Given any recursive tree $T\subseteq 2^{<\omega}$ so that $[T]$ only contains Martin-L\" of random reals, there is a weakly-Demuth-test $\{W_{f(i)}\}_{i\in\omega}$ so that for any $\sigma\in 2^{<\omega}$, if $[\sigma]\cap [T]\neq \emptyset$, then $[\sigma]\cap [T]\cap (\bigcap_i W_{f(i)})\neq \emptyset$.
\end{lemma}
\begin{proof}
The proof is quite similar to the one by Yu for weak-2-randomness.

Given a  recursive tree $T\subseteq 2^{<\omega}$ so that $[T]$ only contains Martin-L\" of random reals, we try to build a weakly-Demuth-test $\{W_{f(i)}\}_{i\in\omega}$ so that $W_{f(i)}$ densely meets $[T]$ for every $i$.  Suppose $\mu(T)>2^{-r}$ for some number $r$.

To build $W_{f(i)}$,  for each $\sigma$, we try to search some $\tau\succ \sigma$ with $|\tau|=|\sigma|+^{\lceil}\sigma^{\rceil}+i+1$ so that $[\tau]\cap [T]\neq\emptyset$ where $^{\lceil}\sigma^{\rceil}$ is the G\" odel number of $\sigma$. We put such $\tau$'s into $W_{f(i,0)}$.  The searching way is ``from left to right". In the most lucky case, we don't make mistake, then $W_{f(i,0)}$ densely meets $[T]$ and $\mu(W_{f(i,0)})\leq 2^{-i-1}$. But of course we may make mistakes. Note that ``making mistakes" is a $\Sigma^0_1$-sentence. Once we find the measure of mistakes greater than $2^{-i-1}$, we change $W_{f(i,0)}$ to be $W_{f(i,1)}$. By some tricks, we can ensure not to make the same mistakes. This means the times of ``making big mistakes" is no more than $2^{i+1}$ times.

So $f(i)$ change at most $2^i$ times and $\mu(W_{f(i)})\leq 2^{-i}$.
\end{proof}

The left part is the exactly same as weak-2-randomness case.

Let $\mathbb{P}=(\mathbf{T},\leq)$ be partial ordering so that every $T\in \mathbf{T}$ is a recursive tree  and only contains Martin-L\" of random reals. $T_1\leq T_2$ if and only if $T_1\subseteq T_2$. Clearly every $\mathbb{P}$-generic real is weakly-2-random and so weakly-Demuth-random.

\begin{lemma}\label{lemma: Yu-dense-avoid}
Given any $\mathbf{\Pi^0_2}$ set $G$ only containing weakly-Demuth-random reals, the set $\mathcal{D}_G=\{P\in \mathbf{T}\mid P\cap G=\emptyset\}$ is dense.
\end{lemma}
\begin{proof}
Immediately from Lemma \ref{Lemma: Yu-meet-dense-by-weakly-Demuth-test}.
\end{proof}

So for any $\mathbf{\Sigma^0_3}$ set $H$ only containing  weakly-Demuth-random reals, by Lemma \ref{lemma: Yu-dense-avoid}, any sufficiently $\mathbf{P}$-generic real doesn't belong to $H$.

I don't know whether  the collection of Demuth-random reals is $\mathbf{\Sigma}^0_3$

\part{Traceability}

 \section{Oct 2011: There are $2^{\aleph_0}$-many $n^{1+\epsilon}$-jump traceable reals}
 
 By Yu Liang.
This is   joint work with Denis Hirschfeldt.

\begin{definition} Let $h$ be an order function. We say that a real $x$ is {\it $h$-jump traceable}  if for each Turing functional  $\Phi$,  there is an $h$-bounded c.e.\ trace $(U_n)\sN n$ such that $\Phi^x$ total implies $\Phi^x(n) \in U_n$ for a.e.\ $n$.  \end{definition} 
\begin{thm} 
Let  $\epsilon$ be  an arbitrary small positive rational.  Then there are $2^{\aleph_0}$-many $n^{1+\epsilon}$-jump traceable reals. \end{thm}

We construct a perfect tree $T$ so that  there is a uniformly r.e. sequence $\{U_e\}_{e\in \omega}$ such that 
\begin{itemize}
\item For every $e$, $|U_e|\leq e^{1+\epsilon}$;
\item for any $x\in [T]$ and any $e$, if $\Phi_e^x$ is total, then $\Phi_e^x(n)\in U_n$ for any $n\geq e$.
\end{itemize}

First we show that for a single index $e$,  there is a uniformly r.e. sequence $\{U_e\}_{e\in \omega}$ such that 
\begin{itemize}
\item For every $e$, $|U_e|\leq e^{1+\epsilon}$;
\item for any $x\in [T]$, if $\Phi_e^x$ is total, then $\Phi_e^x(n)\in U_n$ for any $n$.
\end{itemize}

Let $\delta(n)$ be a computable increasing function so that $$(\delta(n))^{1+\epsilon}>2^{n+1}\sum_{i\leq n}\delta(i).$$

We build an embedding function $T:2^{<\omega}\to 2^{\omega}$ by approximation.

At stage $0$, $T_0$ is identity. 

At stage $s+1$, if there is some $m\in [\delta(n),\delta(n+1))$ and some finite string $\sigma$,  where we may assume that $|\sigma|> n$,  so that $\Phi_e^{T_s(\sigma)}(m)\downarrow$ at stage $s$. Let $T_{s+1}[\sigma\uh n]$ be $T_s[\sigma]$ and move other values corresponded to remain $T_{s+1}$ to be an embedding function.  In other words, we kill all the branching nodes up to $T_s(\sigma)$ extending $T_s(\sigma\uh n)$ to narrowing the possible values of $\Phi_e$. Put $\Phi_e^{T_s(\sigma)}(m)$ into $U_m$.

The intuition behind the construction  is to make all the values between $[\delta(n),\delta(n+1))$ be ``on  some single node".

By the finite injury,  $\lim_{s\to \infty}T_s$ exists.

Note that for each $m\in [\delta(n),\delta(n+1))$, we put at most $\delta(n)\cdot 2^{n+1}$, which is not greater than $m^{1+\epsilon}$,  many values into $U_m$.

\bigskip

For the general case, there is no essential difference. Just let $\delta$ be even faster so that we may narrow the possible branches.

\begin{conjecture}
There are only countably many id-jump traceable reals.
\end{conjecture}

\newpage

\section{ December 2011: A c.e.\ $K$-trivial which is not $o(\log x)$ jump traceable.}

By Turetsky. Paper accepted in Information Processing Letters.

\begin{lemma}
If~$h$ is an order with $\sum_{x = 0}^\infty 2^{-d\cdot h(x)} = \infty$ for all $d > 0$, then for any~$c > 0$, there is an~$n > 0$ with $h(2^{cn}) \leq n$.
\end{lemma}

\begin{proof}
Fix a~$c$, and suppose $h(2^{cn}) > n$ for all~$n$.  Then $h(x) > \frac{1}{c}\log x$ for all~$x$, and thus $\sum_{x = 0}^\infty 2^{-2c\cdot h(x)} < \infty$, contrary to hypothesis.
\end{proof}

\begin{theorem}
There is a c.e.\ $K$-trivial set~$A$ which is not jump traceable at any order~$h$ with $\sum_{x = 0}^\infty 2^{-d\cdot h(x)} = \infty$ for all $d > 0$.  In particular, it is not jump traceable at any order~$h \in o(\log x)$.
\end{theorem}

\begin{proof}
We make~$A$ $K$-trivial in a standard fashion: fix~$W$ a c.e.\ operator such that $\lambda([W^X]) \leq 1/2$ and $[W^X]$ contains all $X$-randoms for all oracles~$X$ (e.g., $[W^X]$ is an element of the universal oracle ML-test); we shall construct a c.e.\ set~$V$ with $[W^A] \subseteq [V]$ and $\lambda([V]) < 1$.  Our construction of~$V$ is the obvious one: whenever we see a string $\sigma \in W^A[s]$, we enumerate~$\sigma$ into~$V$.  We will ensure that $\lambda([V]) < 1$ by appropriate restraint on~$A$.

Let~$h^e$ be an enumeration of all (partial) orders, and let~$\seq{T_k^e(x)}{k \in \omega}$ be an enumeration of all $h^e$ bounded c.e.\ traces.  We construct a partial $A$-computable function~$f^A$, and (assuming~$h^e$ satisfies the hypothesis of the theorem) make~$A$ not jump traceable at order $h^e$ by meeting the following requirements:
\begin{description}
\item[$P_k^e$] $(\exists x) [ f^A(x)\!\!\downarrow \ \& \ f^A(x) \not \in T_k^e(x)]$.
\end{description}

We partition the domain of~$f^A$ into infinitely many sets~$B^e$, and work to meet requirements for~$h^e$ on~$B^e$.  However, our choice of pairing function matters: each~$B^i$ must be an arithmetic progression.  So we let $B^e = \{ x\cdot 2^{e+1} + 2^e \mid x \in \omega\}$.

The basic $P_k^e$-strategy is straightforward.  Choose an~$x \in B^e$, wait until~$h^e(x)\!\!\downarrow$, and then define $f^A(x)$ to a large value with large use.  Wait until $f^A(x) \in T_k^e(x)$.  Change~$A$ below the use and redefine~$f^A(x)$ to a large value.  Eventually we win, since~$T_k^e(x)$ has bounded size.

The complication comes in the interaction between positive requirements and ensuring that~$\lambda([V]) < 1$ --- redefining~$f^A(x)$ may cause measure to leave~$[W^A]$.  As long as the total measure which leaves~$[W^A]$ over the course of the construction is strictly less than~$1/2$, we are fine.  $P_k^e$ must be careful to only contribute a small fraction of that.

Suppose that for some constant~$c$, $P_k^e$ has claimed for its future use $2^{c\cdot n}$ many elements of~$B^e$ on which~$h^e$ takes value less than~$n$, for some~$n$.  $P_k^e$ will follow the basic strategy on the first of these elements, but will discard this and choose a new one if the cost of clearing the computation rises too high.  What is ``too high''?  This will depend on how much progress~$P_k^e$ has made on the current witness.  Initially (if $P_k^e$ has not yet cleared any computations on this box),~$P_k^e$ will discard the box if the measure lost by clearing the computation rises beyond~$2^{-c\cdot n}$.  Each time~$P_k^e$ completes a loop of the basic strategy (that is, each time it clears a computation and the current $n$-box is promoted), its threshold is multiplied by~$2^{c-1}$.  So if the box has been promoted $m$ times,~$P_k^e$ will discard it if the measure lost by clearing the computation rises beyond~$2^{-c\cdot n + (c-1)\cdot m}$.  When it discards a box and begins working with the next one,~$P_k^e$'s threshold returns to $2^{-c\cdot n}$

Now, let us analyze this strategy in the absence of action for other positive requirements.  Assume that $c > 1$.  Because uses are always chosen large, the measure which restrains a given $n$-box of~$P_k^e$ is always disjoint from the measure which restrains a different $n$-box.  So if $x_1, \dots, x_l$ are the $n$-boxes~$P_k^e$ is restrained on, and box~$x_i$ is restrained after~$m_i$ promotions, we know
\[
\sum_{i = 1}^l 2^{-c\cdot n + (c-1)\cdot m_i} \leq 1/2.
\]
Since $m_i \geq 0$, we know $l \leq 2^{c\cdot n}$.  Actually, since the measure which restrains~$x_i$ must be strictly more than the threshold, we know that $l < 2^{c\cdot n}$.

Now, all the promotions of~$x_i$ caused the discarding of at most
\[
\sum_{j = 0}^{m_i-1} 2^{-c\cdot n + (c-1)\cdot j} \leq 2^{-c\cdot n + (c-1)\cdot m_i - c + 2}
\]
measure.  So the total measure discarded by all the~$x_i$ is at most
\[
\sum_{i = 1}^l 2^{-c\cdot n + (c-1)\cdot m_i - c + 2} = 2^{-c+2}\sum_{i = 1}^l 2^{-c\cdot n + (c-1)\cdot m_i} \leq 2^{-c+1}.
\]

Since $l < 2^{c\cdot n}$, we know there is some~$x_{l+1}$ which is never restrained, and thus causes the~$P_k^e$ strategy to succeed.  Now, let us analyze the total measure discarded by the promotions of~$x_{l+1}$.  It is at most
\[
\sum_{j = 0}^{n} 2^{-c\cdot n + (c-1)\cdot j} \leq 2^{-c\cdot n + (c-1)\cdot n + 1} = 2^{-n + 1}.
\]

So the total discarded measure caused by~$P_k^e$ is at most $2^{-n+1} + 2^{-c+1}$.

We now integrate all the positive strategies via finite injury.  When each~$P_k^e$ strategy is initialised, it chooses a large integer~$r$.  $P_k^e$ then searches for~$2^{c\cdot n}$ many unclaimed elements of~$B^e$ with $h^e \leq n$ on these elements and $2^{-n+1} + 2^{-c+1} < 2^{-(r+3)}$.  By the previous lemma, if~$h^e$ satisfies the hypothesis of the theorem, since~$B^e$ is an arithmetic progression,~$P_k^e$ will eventually find such elements.  It then begins running the above strategy.  Whenever a higher priority strategy enumerates an element into~$A$,~$P_k^e$ is initialised, choosing a new~$r$ and searching for new elements to work on.  By the usual finite injury argument, every strategy eventually succeeds.  Further, the lost measure is at most $\sum_{r = 0}^\infty 2^{-(r+3)} = 1/4$.  So $\lambda([V]) \leq 3/4$.
\end{proof}

 \newpage

%%%%%%%%%%%%%%%%%%%%%%%%%%%%%%%%%%%%%%%%
\part{Randomness and computable analysis}

 \section{April 2011: Randomness and Differentiability}

\def\PiOne{\ensuremath{\mathrm{\Pi}^0_1}\xspace}

 by    Nies (mainly), Bienvenu, Hoelzl, Turetsky and others.

	Brattka, Miller and Nies have submitted  their paper entitled Randomness and Differentiability \cite{Brattka.Miller.ea:nd}. The main thesis of the  paper is that algorithmic randomness of a real is equivalent to differentiability of effective functions at the real. It goes back to earlier work of Demuth, for instance~\cite{Demuth:75},  and Pathak~\cite{Pathak:09}. 

	For  most  major algorithmic randomness notions, on can now  provide a class of effective functions on the unit interval   so that 

	\vsp
	\n ($*$)  \  a real $z\in [0,1]$ satisfies the randomness notion  $\LLR$ 

	\hspace{3cm} each function in the class is differentiable at $z$.

	The matching between algorithmic randomness notions and classes of effective functions  is  summarized in Figure~\ref{fig:diagram}. 

	\tikzstyle{header} = [font=\bfseries]
	\tikzstyle{notion} = []
	\tikzstyle{class} = [text centered]
	\tikzstyle{comp} = [midway, right=-0.1cm, text centered, text width=2.4cm, font=\small\itshape]

	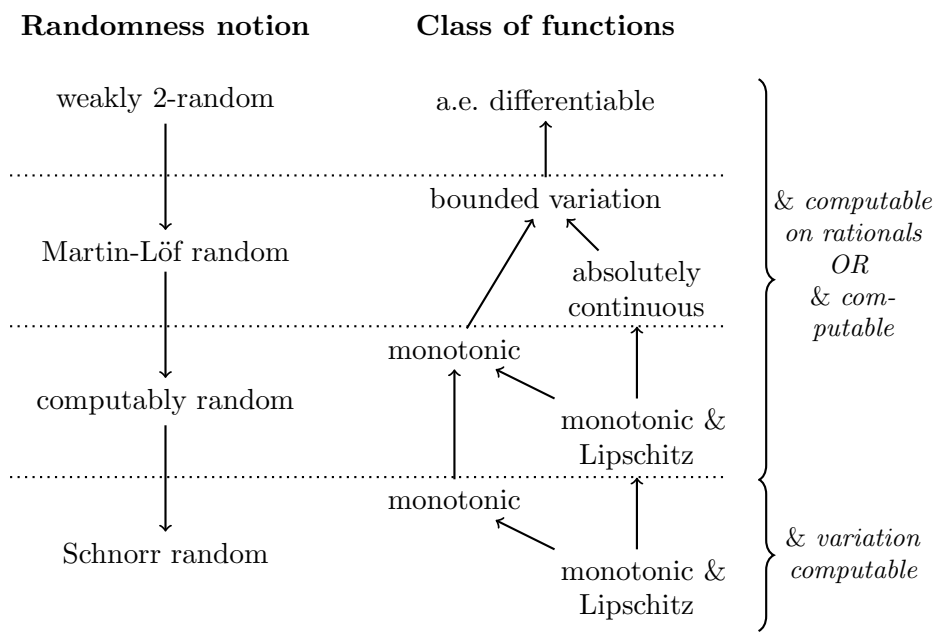
\begin{figure}[htbp] 
	% \bc \scalebox{0.8}{\includegraphics{figures/Function_randomness_notions_3.pdf}} \ec
	\begin{tikzpicture}[node distance=2cm, thick, segment amplitude=0.2cm]
		\node[header] (rn) {Randomness notion};
		\node[notion, below of=rn, yshift=1cm] (w2r) {weakly $2$-random};
		\node[notion, below of=w2r] (mlr) {\ML\ random};
		\node[notion, below of=mlr] (cr) {computably random};
		\node[notion, below of=cr] (sr) {Schnorr random};
		\path[->] (w2r) edge (mlr);
		\path[->] (mlr) edge (cr);
		\path[->] (cr) edge (sr);

		\node[header, right of=rn, xshift=3cm] (cf) {Class of functions};
		\node[class, below of=cf, yshift=1cm] (aed) {a.e.\ differentiable};
		\node[class, below of=aed, yshift=0.7cm] (bv) {bounded variation};
		\node[class, below of=bv, xshift=1.2cm, yshift=0.8cm, text width=2cm] (ac) {absolutely continuous};
		\node[class, below of=bv, xshift=-1.2cm] (m1) {monotonic};
		\node[class, below of=ac, text width=2cm] (mL1) {monotonic~\& Lipschitz};
		\node[class, below of=m1] (m2) {monotonic};
		\node[class, below of=mL1, text width=2cm] (mL2) {monotonic~\& Lipschitz};

		\path[->] (bv) edge (aed);
		\path[->] (ac) edge (bv);
		\path[->] (m1) edge (bv);
		\path[->] (mL1) edge (ac);
		\path[->] (mL1) edge (m1);
		\path[->] (m2) edge (m1);
		\path[shorten >=-1pt, ->] (mL2) edge (mL1);
		\path[->] (mL2) edge (m2);

		\draw[snake=brace, segment amplitude=0.2cm] ([xshift=2.8cm]aed.north) -- ([xshift=1.6cm]mL1.south) node[comp] {$\&$ computable on~rationals\break OR \break $\&$ computable\break};
		\draw[snake=brace] ([xshift=4cm]m2.north) -- ([xshift=1.6cm]mL2.south) node[comp] {$\&$ variation computable};

		\draw[dotted] ([yshift=-2cm]rn.west) -- ([yshift=-2cm, xshift=0.5cm]cf.east);
		\draw[dotted] ([yshift=-4cm]rn.west) -- ([yshift=-4cm, xshift=0.5cm]cf.east);
		\draw[dotted] ([yshift=-6cm]rn.west) -- ([yshift=-6cm, xshift=0.5cm]cf.east);
	\end{tikzpicture}
	\caption{Randomness notions matched with classes of effective functions defined on $[0,1]$ so that ($*$) holds}
	\label{fig:diagram}\end{figure}

Groups currently working in this area include  

\bi\item Freer, Kjos-Hanssen, Nies: computable Lipschitz function. In prep.

\item J. Rute (Carnegie Mellon, student of J.\  Avigad): higher dimensions, measures. Characterized Schnorr randomness.  In preparation.

\item Pathak, Rojas and Simpson: Characterized Schnorr randomness. Submitted Nov 2011.

\item Kenshi Myabe (Kyoto U): a view via function tests (integral tests). Characterized Schnorr randomness, Kurtz ``randomness''. Submitted Nov 2011.
\ei

\subsubsection{Notation}

For  a   function $f$,  the \emph{slope} at a pair $a,b$  of distinct reals in its domain  is 
   
   \[ S_f(a,b) = \frac{f(a)-f(b)}{a-b}\]
   
   Recall that if $z$ is in the domain of $f$ then  \begin{eqnarray*} \ol D f(z)  & = & \limsup_{h\ria 0} S_f(z, z+h) \\
   \ul D f(z) & = & \liminf_{h\ria 0} S_f(z, z+h)\end{eqnarray*}

$f'(z)$ exists just if those are equal and finite.

\subsection{General thoughts} % (fold)
\label{sub:general_thoughts}

Why do algorithmic rdness notions for reals match so well with analytic notions for functions in one variable? The functions have to be computable (or even  computable in the variation norm), but then, most functions defined  on the unit interval that arise in analysis, such as $e^x$ or $\ln \sqrt {x+1} $,  are computable. (If the derivative is computable the function is even variation computable.) 

Analytic notions on functions have been studied   systematically   by Lebes\-gue~\cite{Lebesgue:1909} and even earlier. Algorithmic randomness notions have been studied for 45 years, starting with \cite{Martin-Lof:66} (see  \cite{Muchnik.ea:98} for the Russian side of the story).

\subsubsection{How about matching the remaining randomness notions?} % (fold)
\label{ssub:subsubsection_name}

% subsubsection subsubsection_name (end) 
\begin{question} Can one match the notions of partial computable randomness and permutation randomness (see \cite[Ch.\ 7]{Nies:book}) with a class of real functions in the sense of ($*$)? \end{question} 

	First of all we would like a proof that these notions are actually about real numbers (not bit sequences). Just like for computable randomness \cite[Section 4]{Brattka.Miller.ea:nd} one would need to show that they are independent of the choice of a base  (which is  2 in the definition). 

	\begin{question}
		How about matching 2-randomness? Demuth randomness? 
	\end{question}

	Demuth's paper \cite{Demuth:88} went a long way towards answering   the second part of the question: at a Demuth random real $z$,  the Denjoy (French, 1907) alternative  (DA) holds for each Markov computable (=constructive) function $f$.     See Subsection~\ref{ss:Denjoy} for definitions, and Subsection \ref{ss:DA_Markov_computable} for a result stronger than Demuth's.

	\subsubsection{Are there   function notions from analysis that can't be matched?} % (fold)
	\label{ssub:are_there_other_function_notions_}

\

This remains open.  The analytic notions  in Figure~\ref{fig:diagram} seem to be the major ones. 

	% subsubsection are_there_other_function_notions_ (end)

\subsubsection{The case of ML-randomness} 	Recall that a function $f \colon \,  [0,1]\ria \mathbb R$ is   of \emph{bounded variation} if 

	\[ \infty > \sup \sum_{i=1}^n   | f(t_{i+1}) - f(t_i)|, \]
	where the sup is taken over all partitions $ t_1 < t_2 <  \ldots <   t_n$ in $[0,1]$.

The characterization of ML-randomness is via differentiability of computable functions of bounded variation. A direct proof is probably in Demuth \cite{Demuth:75}. In \cite{Brattka.Miller.ea:nd} we get the harder direction   $\RA$  in $(*)$ above  out of the case for computable randomness via the low basis theorem.

   A direct proof for $\RA$ is not available currently. Here is  a simple proof of a weaker statement: 

\begin{fact} Let $z$ be ML-random. If $f$ is a computable function of bounded variation, then $\ol D f(z)< \infty$. \end{fact}

\begin{proof} We may assume that the variation $Var(f)\le 1$.  A \emph{dyadic $n$-interval} inside $[0,1]$ has the form $(i \tp{-n}, k\tp{-n})$ where $i<k$ are naturals.  Let  $G_r$ be the union of all intervals  $(x,y)$   such  that $S_f(x,y) >  2^r$ and  that are dyadic $n$ intervals for some $n$. 	
	 We claim that  $\leb G_r \le \tp{-r}$. For let $G_{r,n}$ be the union of dyadic $n$-intervals in $G_r$. Consider the partition consisting of  the  endpoints of the maximal dyadic $n$-intervals contained in $G_{r,n}$. Since $V(f) \le 1$ we can conclude that $\leb G_{r,n} \le \tp {-r}$.  Since $G_r = \bigcup_n G_{r,n}$ this shows $\leb G_r \le \tp{-r}$.

	Now clearly $(G_r) $ is uniformly c.e., hence a ML-test. If $\ol D f(z) = \infty $ then (because $f$ is continuous) $z$   fails this test. 
\end{proof}

%Maybe Demuth shows that all computable functions of b.v. satisfy Denjoy alternative at a ML-random? 

% subsubsection general_thoughts (end)

\subsubsection{A simpler proof of a special case of \cite[Theorem 4.1(iii)$\to$(ii)]{Brattka.Miller.ea:nd}} 

For a nondecreasing function $f\colon [0,1] \to \RR$ and a real $r$ let $M^f_r $ be the (dyadic) martingale associated with the slope $S_f$ evaluated at intervals of the form $[ r + i \tp{-n}, r+ (i+1) \tp{-n}]$. Solecki, combining  ideas from his paper \cite{Morayne.Solecki:89} with the middle thirds lemma \cite[2.5]{Brattka.Miller.ea:nd},  has shown the following (personal  communication):

\begin{thm} Suppose   $f$ is a   nondecreasing function  which is not differentiable at $z\in [0,1]$.  Suppose that $\ul Df(z)= 0$, or  $\ol Df(z) /\ul Df(z) > 72$. Then one of the     martingales $M^f$, $M^f_{1/3}$ does not converge on $z$.   \end{thm}

\begin{proof}  For an interval $L$ with endpoints $a,b$ we write $S(L)$ for $S_f(a,b)$. Note that $S(I \cup J) \le   \max \{S(I), S(J)\}$ for non-disjoint   intervals $I, J$. 

% We also write $L^-$ for $\inf I$ and $L^+$ for $\sup L$. 

For $m \in \NN$ let $\+D _m $ be the collection of intervals of the form $$[k \tp{-m}, (k+1)\tp{-m}]$$ where $k \in \ZZ$. Let $  \+ D'_m$ be the set of  intervals $(1/3)  +I $ where $I \in \+ D_m$. 

The underlying geometric fact is simple:
\begin{fact} \label{fact:geom} Let $m \ge 1$.  If  $I \in \+ D_m$ and $J \in \+ D'_m$, then the distance between an  endpoint of $I$ and an endpoint of $J$ is at least $1/(3 \cdot 2^m)$.
\end{fact}
To see this: assume  that  $(k \tp{-m} - ( p \tp{-m} +1/3) < 1/(3 \cdot 2^m)$. This yields $3k-3p-2^m)/ (3 \cdot2^m) < 1/(3 \cdot 2^m)$, and hence $3| 2^m$, a contradiction.
  
  \begin{claim}   \label{claim:slope below}  Let $z \in I \cap J$ for intervals $I \in \+ D_m, J \in \+ D'_m$.
 Suppose $z \in L$ for some interval $L$ of length $d$, where  
   $\tp{-m}/3 \ge d \ge \tp{-m-1}/3$. Then %$L \sub I \cup J$ and 
  \begin{equation} \label{eqn:slope_L_1} S(L)/12 \le  \max \{S(I) , S(J)\}. \end{equation}
  \end{claim}

  Clearly by Fact~\ref{fact:geom}  we have $L \sub I \cup J$.  Let $I \cup J = [a,b]$. Since $f$ is nondecreasing and $b-a \le 12  |L|$,   we have  
  
  $S(L) \le  (f(b)- f(a))/|L|   \le 12 S(I \cup J) \le 12 \max \{S(I), S(J)\}$.  This shows the claim.

We now give a dual claim where we need the middle thirds.
  
  \begin{claim}   \label{claim:slope above}  Let $z \in I \cap J$ for intervals $I \in \+ D_m, J \in \+ D'_m$.
 Suppose $z $ is in the middle third of some interval $ L$  of length $d$, where  
  $\tp{-m+1} \ge d/3 \ge \tp{-m}$. Then %$L \sub I \cup J$ and 
  \begin{equation} \label{eqn:slope_L_2} 6S(L) \ge  \max \{S(I) , S(J)\}. \end{equation}
  \end{claim}

We only prove it for $I$.  Note that $I \sub L$ because $z$ is in the middle third of $L$. Since $d \le 6|I|$,  and $f$ is nondecreasing, we obtain  $6S(L) \ge  S(I)$ as required.  
This shows the claim.

Now we argue similar to   \cite{Brattka.Miller.ea:nd}.  By the hypothesis that $f'(z)$ does not exist and the middle thirds lemma \cite[2.5]{Brattka.Miller.ea:nd}, we can choose   rationals $\wt \beta < \wt  \gamma$ such that  $\wt \gamma / \wt \beta > 72 $ and

 \begin{eqnarray*}  \wt \gamma  &   <  &   \lim_{h \ria 0} \sup \{ S_f(x,y)\colon \,  0 
 \le y-x \le h \lland  z  \in    (x,y)  \},  \\
\wt \beta  &   > &    \lim_{h \ria 0} \inf \{ S_f(x,y )\colon \,  0 
 \le y-x \le h \lland    z  \in  \text{  middle third of}  \  (x,y) \} \end{eqnarray*}

By definition we can choose arbitrarily short intervals $L$ containing $z$ such that $S(L) \ge \wt \gamma$, and  arbitrarily short intervals $L$ containing $z$ in their middle thirdy such that $S(L) \le \wt \beta$. Then,  by Claim~\ref{claim:slope below}, for one of the $\+D$ type  or the  $\+ D'$ type intervals, there are arbitrarily short intervals $K$ of this type such that $z \in K$ and $S(K) \ge \wt \gamma/12$.   By Claim~\ref{claim:slope above}, for {\it both} types of  intervals, there are arbitrarily short intervals $K$ of this type such that $z \in K$ and $S(K) \le \wt 6\beta $.  Since   $\wt \gamma / \wt \beta > 72 $, this means that one of the     martingales $M^f$, $M^f_{1/3}$ does not converge on $z$. 
\end{proof}

%%%%%%%%%%%%%%%%%%%%%
\subsection{Denjoy alternative, even for partial functions}
\label{ss:Denjoy}

Our version of the  Denjoy alternative for a function $f$ defined on the unit interval says that
	\begin{equation} \label{eqn:DA} \text{either $f'(z)$ exists, or $\ol Df(z) = \infty $ and $\ul Df(z) = -\infty$}. \end{equation}

 It is a consequence of the  classical  Denjoy (1907) -Young- (1912)  - Saks (1937)  theorem    that for \emph{any} function  defined on the unit interval, the  DA holds at almost every $z$.  The actual result  is  in terms of  right and left upper and lower Dini derivatives denoted $D^+f(z)$ (right upper) etc.     Denjoy proved it for continuous, Young for measurable and Saks for all functions. 
 
   The result  is used for instance to show that $f'$ is always Borel (as a partial function). A paper by Alberti, Csornyei,  Laczkovich, and Preiss (Real Anal. Exchange Vol. 26(1), 2000/2001, pp. 485-488)  revisits the  DA. 

\begin{definition} \label{def:compreal} A \emph{computable real} $z$  is given by a computable Cauchy name,  i.e., a  sequence $(q_n)\sN n$ of   rationals converging to $z$ such that $|q_{n+1} - q_n | \le \tp{-n-1}$. Then $|z- q_n| \le \tp{-n}$.  If the Cauchy name is understood we sometimes write $(z)_n$ for $q_n$. 
\end{definition} 
 
Recall the following. 

 \begin{definition} \label{def:Markov_computable}  A function $g$ defined on the computable reals    is called \emph{Markov computable} if from a computable Cauchy name for $x$ one can compute a computable  Cauchy name for $g(x)$.
\end{definition}

Most relevant   work of  Demuth  on the Denjoy alternative for effective functions    is  in this \texttt{http://dl.dropbox.com/u/370127/DemuthPapers/Demuth88PreprintDenjoySets.pdf}{preprint}, simply  called \emph{Demuth preprint} below. This later turned into the  paper \texttt{http://dl.dropbox.com/u/370127/DemuthPapers/Demuth88PaperDenjoySets.pdf}{Remarks on Denjoy sets}; unfortunately a lot of things  from the preprint are missing there.
    
  Since the functions aren't total any more, we have to introduce ``pseudo-derivatives'' at $z$, taking the limit of slopes close to $z$ where the function is defined.      This is presumably what Demuth did. 
   Consider  a function $g$ defined on $I_\QQ$,  the rationals in  $[0,1]$. For  $z \in [0,1]$ let
    
    \vsp
    
    $\utilde  Dg(z) = \liminf_{h \to 0^+}  \{S_g(a,b)  \colon a, b \in I_\QQ   \lland  \, a\le x \le b \lland\, 0 <  b-a\le h\}$.
    
       $\widetilde  Dg(z) = \limsup_{h \to 0^+}  \{S_g(a,b)  \colon a, b \in I_\QQ   \lland  \, a\le x \le b \lland\, 0 <  b-a\le h\}$.
       
       \vsp

  Note that Markov computable functions are continuous on the computable reals, so it does not matter which  computable dense set of 
      computable reals we take in the definition of these pseudo-derivatives.  For a total continuous function $g$, we have  $\utilde  Dg(z) = \ul Dg(z)$ and    $\widetilde  Dg(z) = \ol Dg(z)$.  See last section of~\cite{Brattka.Miller.ea:nd}.
       
     Suppose more generally we have   a  function $f$ with domain containing $I_\QQ$ we say that the \emph{Denjoy alternative}  holds if 
	\begin{equation} \label{eqn:pseudoDA} \text{either $\widetilde Df(z) = \utilde Df(z) < \infty$, or $\widetilde Df(z) = \infty $ and $\utilde Df(z) = -\infty$}. \end{equation}
 
       This is equivalent to  (\ref{eqn:DA}) if the function is total and continuous. 
       
   \subsection{Denjoy randomness coincides with computable randomness}
  
  \label{ss:DenjoyComputable_random}     
       
\begin{definition}[Demuth preprint, page 4]  \label{df:DenjoyRandom}  A  real $z \in [0,1]$ is called  Denjoy random (or a Denjoy set) if  for no Markov computable function $g$ we have $\utilde D g(z) = \infty$. \end{definition} 

In the Demuth preprint, page 6,  it is shown that 
if  $z \in [0,1]$ is  Denjoy random, then for every computable $f \colon \, [0,1] \to \mathbb R$ the Denjoy alternative  (\ref{eqn:DA}) holds at~$z$.  
Combining this with  the results in \cite{Brattka.Miller.ea:nd} we can now figure out what Denjoy randomness is, and also obtain  a pleasing new characterization of computable randomness through differentiability of computable functions. 

\begin{theorem}[Demuth, Miller, Nies, \Kuc{}]  \label{thm:DenjoyCR}  The following are equivalent for a real $z \in [0,1]$

\bi 

\item[(i)] $z$ is Denjoy random. 

\item[(ii)]  $z$ is computably random

\item[(iii)]  for every computable $f \colon \, [0,1] \to \mathbb R$ the Denjoy alternative  (\ref{eqn:DA}) holds at~$z$.    \ei \end{theorem}

\begin{proof}  
\n (i)$\to$ (iii) is the result of Demuth. 

\vsps

\n (iii)$\to$(ii) Let $f$ be a nondecreasing computable function. Then $f$ satisfies the Denjoy alternative at $z$. Since $\ul Df(z) \ge 0$, this means that $f'(z)$ exists. 

This implies that  $z$ is computably random by \cite[Thm.\  4.1]{Brattka.Miller.ea:nd}.

\vsps

\n (ii)$\to$ (i).
Given a binary string $\sss$,  we will  denote  the open interval  $(0.\sss, 0. \sss+ \tp{-\sssl})$  by $(\sss)$. We also write $S_f(\sss)$ to mean   $S_f(a,b)$ where $(a,b) = (\sss)$.

By hypothesis  $z $ is incomputable, and in particular  not a   rational. Suppose that the function  $g$ is Markov computable and $\utilde Dg(z) = + \infty$. Choose dyadic rationals $a,b$ such that  $(a,b) = (\sss)$ for some string  $\sss$,  $z \in (a,b) \sub [0,1]$ and $S_g(r,s) > 4$ for each $r,s$ such that $z \in (r,s) \sub (a,b)$.

Define a computable martingale $M$ on extensions   $\tau \succeq \sss$ that succeeds on (the binary expansion of)~$z$.  In the following $\tau$ ranges over such extensions. 

Firstly, note that   $ S_g(\tau)$  is a computable  real uniformly in $\tau$. Furthermore,   the  function $ \tau \to S_g(\tau)$ satisfies the martingale equality,  and succeeds on $z$ in the sense that its values are unbounded (even converge to $\infty$) along~$z$. However,  this function   may have negative values; J.\  Miller has called this  ``betting with debt'' because we can increase our capital  at a string $\sss 0$ beyond $2 S_g(\sss)$ by incurring a debt, i.e. negative value, at $S_g(\sss 1)$.  We now define  a computable martingale $M$ that succeeds on~$z$ and does not use  betting with debt. 

Let $M(\estring) = S_g(\estring)$. Suppose now that  $M(\tau)$ has been defined and is positive. 
 
 \vsps
 
 \n \emph{Case 1.}  There is $u \in \{0,1\}$ such that, where $v= 1-u$, we have $S(\tau v)_1 < 1$ (this is the second term in the Cauchy name for the computable real $S(\tau v)$, which is at most $1/2$ away from that  real). Then $S(\tau v)< 2$. By choice of $a,b$ we now know that $z $ is not  an extension of $\tau v$. Thus, we let $M$ double its capital along $\tau u$, let $M(\rho) = 0$ for all $\rho \succeq \tau v$. (The martingale  $M$ stops betting on these extensions.)

 \vsps
 
  \n \emph{Case 2.} Otherwise. Then $S(\tau v)>0$ for $v=0,1$. We let $M$ bet with the same betting factors as $S_g$:
  
$$ M(\tau u) = M(\tau) \frac {S_g(\tau u) }{S_g(\tau)}$$
for $u=0,1$.
Note that $M(\tau u) >0$.

If  Case 1 applies to infinitely many initial segments of the binary expansion of  $z$, then  $M$  doubles its capital along  $z$ infinitely often. Since $M$ has only positive values along $z$, this means that  $\lim_n M(z\uhr n)$ fails to exist, whence $z$ is not computably random by the effective version of the   Doob martingale convergence theorem \cite{}. 

 Otherwise, along $z$,  $M$ is eventually in Case 2. So $M$ succeeds on $z$ because $S_g$ does.
\end{proof}

Note that all we needed for the last implication was that $g(q)$ is a computable real uniformly in a rational $q \in I_\QQ$. Thus, in Definition~\ref{df:DenjoyRandom}  we can replace Markov computability of $g$ by this weaker hypothesis. 
\subsection{The Denjoy alternative for functions satisfying effectiveness notions  weaker than computable}
\label{ss:DA_Markov_computable}
\

This is work of Bienvenu, H\"olzl and Nies who met  at the LIAFA in Paris May-June. 

Demuth \cite{Demuth:88} proved that the DA holds at what is now called Demuth random reals, for each Markov computable function. We show that in fact the much weaker notion of difference randomness is enough! Difference randomness was  introduced by  Franklin and Ng  \cite{Franklin.Ng:10}. They showed that it is equivalent to being ML-random and Turing incomplete. 

\subsubsection{Weak 2-randomness yields the DA for functions computable on $I_\QQ$} 
First we review some things from the last section of \cite{Brattka.Miller.ea:nd}.
Let $I_\QQ = [0,1] \cap \QQ$. A function $f \colon \sub [0,1] \to \RR$ is called \emph{computable on $I_\QQ$} if $f(q)$ is defined and  a computable real uniformly in the rational $q$.

 For any rational $p$,  let 
	
	\bc $\utilde C(p) = \{z \colon \,  \fa t >0 \, \ex a,b  [ a \le z \le b \lland 0< b-a \le t \lland   \, S_f(a,b ) < p \}$,  \ec where $t,a,b$ range over rationals.
Since $f$ is computable on $I_\QQ$, the set  
\bc $\{z \colon \,  \ex a,b  \,  [ a \le z \le b \lland 0< b-a \le t \lland   \, S_f(a,b ) < p\}$ \ec
  is a $\SI 1$ set  uniformly in $t$.	Then $\utilde C(p)$ is $\PI 2$ uniformly in $p$. Furthermore, 
\begin{equation} \label{eqn:CDtilde}\utilde Df(z) <p \RA z \in \utilde C(p) \RA \utilde Df(z) \le p.  \end{equation}
	  Analogously we define 
	\bc $\widetilde C(q) = \{z \colon \,  \fa t >0 \, \ex a,b  [ a \le z \le b \lland 0< b-a \le t \lland   \, S_f(a,b ) >  q \}$. \ec
	
Similar observations hold for these sets.

\begin{thm} \label{thm:DA_w2r}  Let $f \colon \sub [0,1] \to \RR$ be computable on $I_\QQ$. Then $f$ satisfies the Denjoy alternative at every weakly 2-random real $z$. \end{thm} 

\begin{proof}  We adapt   the {classical proof}   in \cite[p. 371]{Bogachev.vol1:07} to the case of pseudo-derivatives.   We analyse the arithmetical complexity of exception sets to conclude that weak 2-randomness is enough for the DA to hold. 

   We let $a,b, p , q$ range over  $I_\QQ$.  Recall Definition~\ref{def:compreal}.  For each 
 $r< s$, $r,s \in I_\QQ$, and for each $n \in \NN$, let  
 
 \begin{equation} \label{eqn:Enrs} E_{n,r,s} = \{ x \in [0,1] \colon \, \fa a,b [ r \le a \le x \le b \le s \to S_f(a,b)_0 >  -n+1]\}. \end{equation} 
 
Note that $E_{n,r,s}$  is a  $\PI 1$ class. For every $n$ we have the implications 
 
 \bc $\utilde Df(z) > -n+2 \to \ex r,s \,  [ z \in E_{n,r,s}]  \to \utilde Df(z) > -n. $ \ec
 
To show the DA (\ref{eqn:pseudoDA}) at   $z$, we may assume that  $\utilde D f(z) > - \infty$ or $\widetilde Df(z) < \infty$. If the second condition holds we replace $f$ by $-f$, so we may assume   the first condition holds. Then   $z \in E_{n,r,s}$ for some    $r,s, n$ as above.  Write $E= E_{n,r,s} $. 

% We need to show that  $\widetilde Df(z) = \utilde Df(z) < \infty$. 

For $p< q$, the class $E \cap  \utilde C(p) \cap  \widetilde C(q)$ is $\PI 2$. By (\ref{eqn:CDtilde}) it suffices to show that each such  class is null. For this, we show that  for a.e.\ $x \in E$, we have $\utilde Df(z) = \widetilde Df(x)$. 
 This remaining part of  the argument is entirely  within classical analysis. Replacing  $f$ by $f(x) + nx$, we may assume that for $x \in E$, we have  \bc $\fa a,b [ r \le a \le x \le b \le s \to S_f(a,b)_0 >  1]$.  \ec
 
 Let $f_*(x) = \sup_{a \le x} f(a)$. Then $f_*$ is  nondecreasing on $E$. Let $g$ be an arbitrary  nondecreasing function defined on $[0,1]$ that extends $f_*$. Then by a  classic theorem of Lebesgue, $L(x)  := g'(x)$ exists for a.e.\ $x \in [0,1]$. 
 
% By Theorem \ref{thm:DenjoyCR}, computable randomness (and therefore certainly weak 2-randomness) of $x$ is enough to guarantee the existence of $g'(x)$.
% 
 The following is a definition from classic analysis due to E.P. Dolzhenko, 1967  (see for instance \cite[5.8.124]{Bogachev.vol1:07} but note the typo there).
 \begin{definition} We say that $E $ is \emph{porous at} $x$  via $\epsilon >0$ if for each $\alpha >0 $ there exists $ \beta$ with $ 0 < \beta \le \alpha$ such that $(x-\beta, x+ \beta)$ constains an open interval of length $\epsilon \beta $ that is disjoint from $E$.  We say that  $E$ is porous at $x$ if it is porous  at $x$ via some~$\epsilon$. \end{definition}

 By the Lebesgue density theorem, the points in $E$  at which $E$ is porous form a null set. 
 
 \begin{claim} \label{cl:nonporous_does_it} For each $x \in E$ such that $L(x)$ is defined and $E$ is not porous at $x$, we have $\widetilde Df(x) \le L(x) \le \utilde Df(x) $. \end{claim} 
 
  Since  $\utilde Df(x) \le \widetilde Df(x)$, this establishes the theorem.
  
 To prove the claim, we show $ \widetilde Df(x) \le L(x)$, the other inequality being symmetric.  Fix $\epsilon >0$.  Choose $\alpha > 0$ such that 
 \begin{equation} \label{eqn:uxv} \fa u,v \in E \, [ (u \le x \le v \lland 0< v - u \le \alpha )  \to S_{f_*}(u,v) \le L(x) (1+\epsilon)]; \end{equation}
 furthermore, since $E$ is not porous at $x$, for each $\beta \le \alpha$, the interval $(x-\beta, x+ \beta)$ contains no open subinterval of length $\epsilon \beta$ that is disjoint from $E$. 
 Now suppose that $a,b \in I_\QQ$, $a < x< b$ and   $\beta = 2(b-a) \le \alpha$. There 
 are $u,v \in E$ such that $0\le a-u \le \epsilon \beta$ and $0 \le v-b \le \epsilon \beta$. Since $u,v \in E$ we have $f_*(u) \le f(a) $ and $f(b) \le f_*(v)$.  Therefore $v-u \le  b-a + 2 \epsilon \beta  = (b-a) (1 + 4 \epsilon)$. It follows that 
 
 \[ S_f(a,b) \le \frac {f_*(v)  - f_*(u)}{b-a} \le S_{f_*}(u,v)(1+4\epsilon) \le L(x) (1+4\epsilon)(1+\epsilon). \qedhere\]
 
 \end{proof} 
 
% The core argument of the proof establishes in particular the following statement, which we restate independently as we will reuse it in another context later.
% % of the theorem is restated in the following theorem which we will reuse in a later part of this article.
% 
% \begin{corollary}\label{porous_diff_restate}
% Assume we have a Markov computable function $f$, a $\PiOne$ class $E$ with the property that there is an $n$ such that for all $x \in E$
% %$\delta > 0$ and for all  and all $a,b \in \dom(f)$ with $|a,b|<\delta$ and $x\in [a,b]$ 
% it holds that $\utilde Df(x) > -n$
% and  $y \in E$ that is computably random and not a porousity point of $E$. Then $f$ is differentiable at~$y$. 
% \end{corollary}
%

\subsubsection{Difference randomness yields the Denjoy alternative  for Markov computable functions}

We slightly reformulate the definition of difference randomness by Franklin and Ng \cite{Franklin.Ng:10}.

\begin{definition}  {\rm A \emph{difference test} is given by a $\PI 1$ class $P \sub [0,1]$, together with a uniformly $\SI 1 $ sequence of classes $(U_n)\sN n$ where $U_n \sub [0,1]$, such that ${\leb (P \cap U_n )} \le \tp{-n}$ for each $n$. A real $z$ fails the test if $ z \in P \cap \bigcap_n U_n$, otherwise $z$ passes the test. We say $z$ is difference random if it passes each difference test. } \end{definition}

Franklin and Ng \cite{Franklin.Ng:10} show that \bc  $z$ difference random $\LR$ $z$ is ML-random $\land$ Turing incomplete. \ec

The following direct proof that no left-c.e. real $\alpha$ (such as Chaitin's  $\Om$) is   difference random might be helpful to understand the concept of difference tests. Let $P = [\alpha, 1]$. Let $U_n = [0, (i+1) \tp{-n})$ where $i \in \NN$ is largest such that $i \tp{-n}  < \alpha$.  Then $P, (U_n)\sN n$ is a difference test and $\alpha \in P \cap \bigcap_n U_n$. 

\begin{thm}  Let $f \colon \sub[0,1] \to \RR$ be Markov computable. Then $f$ satisfies the Denjoy alternative at every  difference random real $z$.  \end{thm}

\begin{proof} 
Note that each Markov computable function is computable on $I_\QQ$. We will show that under the stronger condition of Markov computability,  the relevant null sets in the proof of the foregoing Theorem~\ref{thm:DA_w2r} are   effective null sets in the sense of difference randomness.  
The proof  relies   on Lemmas~\ref{lem:extension} and~\ref{lem:porous}, which will be established  in Subsection~\ref{sss:lemma_proofs} below. 

Given $n\in \NN$ and $r<s$ in $I_\QQ$, define the set $E$ as above. As before, we may assume that for $x \in E$, we have $\fa a,b [ r \le a \le x \le b \le s \to S_f(a,b)_0 >  1]$, and hence the  function  $f_*(x) = \sup_{a \le x} f(a)$ is nondecreasing on $E$.

Firstly, we   show that some total nondecreasing extension $g$ of $h= f_*\uhr E$ can be chosen to be computable.   
Since $f$ is Markov computable, we obtain the following. 

\begin{claim} The function $f_*\uhr E$ is computable. \end{claim}  
To see this, recall that $p,q$ range over $I_\QQ$, and  let $f^*(x) = \inf_{q\ge x} f(q)$. If $x\in E$ and $f_*(x) < f^*(x)$ then $x$ is computable: fix a rational $d$ in between these two values. Then $p< x \lra f(p) < d$, and $q> x \lra f(q) > d$. Hence $x$ is both left-c.e.\ and right-c.e., and therefore computable.
Now a Markov computable function is continuous at every computable $x$. Thus $f_*(x)= f^*(x)$ for each $x$ in $E$.

To compute $f(x)$ for $x \in E$ up to precision $\tp{-n}$, we can now simply search for rationals $p<x <q$ such that $0< f(q)_{n+2} - f(p)_{n+2} < \tp{-n-1}$, and output $f(p)_{n+2}$. If during this search we detect that  $x \not \in E$, we stop. This shows the claim.

\begin{lemma} \label{lem:extension} 
Let $h\colon \sub [0,1] \rightarrow \mathbb{R}^+_0$ be a   computable function that is defined  and non-decreasing on a $\PI 1$  class $E$ . Then $h\uhr   E$ can be extended to a function $g\colon [0,1] \rightarrow \mathbb{R}^+_0$ that is computable and non-decreasing on $[0,1]$.
\end{lemma}

By \cite[Thm.\  5.1]{Brattka.Miller.ea:nd} we know that $L(x) : = g'(x)$ exists for each computably random (and hence certainly each ML-random) real $x$. 

 To show that  in fact $\utilde Df(z) = \widetilde Df(z) = L(z)$ for each difference random real $z$, we need the following.

\begin{lemma} \label{lem:porous}  Let $E \sub [0,1]$ be a $\PI 1$ class. If $z \in E$ is difference random, then $E$ is not porous at $z$. \end{lemma}
Given the lemma, we can conclude the proof of the theorem by invoking Claim~\ref{cl:nonporous_does_it}. \end{proof}

\subsubsection{Proving the two lemmas}  \label{sss:lemma_proofs}

%############### START RUPERT ##################
\def\qt#1{``#1"}

\def\scrC{\ensuremath{E}\xspace}
\def\eps{\ensuremath{\varepsilon}\xspace}
\begin{lemma}\label{approx_lemma}
Let $E$  be a nonempty  \PiOne class. Let $h\colon \subseteq[0,1] \rightarrow \mathbb{R}^+_0$ be a   computable function with domain containing $E$. Then $\sup_{x \in E}\{h(x)\}$ is right-c.e.\ and $\inf_{x \in E}\{h(x)\}$ is  left-c.e. uniformly in an index for $E$.
\end{lemma}
\begin{proof}
We proof the statement for the supremum; the proof for the infimum is analogous. We use the signed digit representation of reals, that is every real is represented by an infinite sequence in $\{-1,0,1\}^\infty$.

%We interpret the elements $x$ in \scrC as the infinite paths of a computable tree. 
We run in parallel an enumeration of $E^c$ and for all $x \in [0,1]$ (given by a Cauchy name $(x_n)_{n\in\mathbb{N}}$) the computations of $h(x)$ up to precision $2^{-n}$. That is, we want to compute $(h(x))_n$, the $n$-th entry of a Cauchy name for $h(x)$. 

Due to uniform continuity, for each $n$  there is a number $n^\prime$ such that for all $x$ in order to accomplish the computation it suffices to have access to the initial segment $(x_0,\dots,x_{n^\prime})$ of the Cauchy name of $x$. When the computation of $(h(x))_n$ halts for some $x$ it also halts for all other $x^\prime$ which have a Cauchy name that  begins with $(x_0,\dots,x_{n^\prime})$, since the computation is clearly the same. We do not know $n^\prime$, so we build a tree of computations that branches into three directions ($-1$, $0$ and $1$) whenever we access a new entry of the Cauchy name of the input. We remove a branch of the tree when it gets covered by $E^c$. Due to the existence of $n^\prime$ the tree will remain finite.

%Clearly by assumption for every $x$ one of the two procedures will halt; and in fact we claim that there exists a finite number of steps~$s_n$ such that after $s_n$ steps this will have happened for all $x\in[0,1]$. The reason is that  The points $x$ covered by such a halting computation form an open set.  
%e use principle which implies that the convergence of $f(x,\eps)$ does not occur on indivual sequences $x$ but on whole open sets of $x$'s. 
%Since \scrC is compact, every infinite set of open sets covering \scrC contains a finite subset of open sets already covering \scrC. 

%Wait until a stage $s_\eps$ where one of these two possibilities mentioned above has happened to all $x \in \scrC[s]$. From this stage on $\sup_{x \in \scrC}\{f(x)\}[s]$ can only decrease as more and more $x$'s leave $\scrC[s]$.

Write $\sup_{x \in E}\{h(x)\}[n]$ for the approximation to the value $\sup_{x \in E}\{h(x)\}$ that we achieve when we proceed as described with precision level $2^{-n}$. If we increase the precision, more of $E^c$ may get enumerated before halting has occured everywhere on the tree; so we see that the sequence $(\sup_{x \in E}\{h(x)\}[n]+2^{-n})_{n\rightarrow \infty}$ is a right-c.e.\ approximation to $\sup_{x \in E}\{h(x)\}$.
%Define the sequence $(s_i)_{i\in N}$ by stipulating that $s_i$ be the first stage where one of these two options has occured for a set of $x$ representing a measure of at least $1-2^{-i}$.
%To get better and better approximations to $\sup_{x \in \scrC}\{f(x)\}$ we output the values $\sup_{x \in \scrC}\{f(x)\}[s_i]$ for increasing $i$.
%
\end{proof}

%\begin{lemma}
%Let $f\colon [0,1] \rightarrow \mathbb{R}^+_0$ be a partial computable function that is total and non-decreasing on a \PiOne class \scrC. Then $f|\scrC$ can be extended to a function $g\colon [0,1] \rightarrow \mathbb{R}^+_0$ that is computable and non-decreasing on $[0,1]$.
%\end{lemma}
%\begin{lemma} \label{lem:extension} 
%Let $h\colon \subseteq [0,1] \rightarrow \mathbb{R}^+_0$ be a   computable function that is defined  and non-decreasing on a $\PiOne$  class $E$. Then $h\ut_E$ can be extended to a function $g\colon [0,1] \rightarrow \mathbb{R}^+_0$ that is computable and non-decreasing on $[0,1]$.
%\end{lemma}
\begin{proof}[Proof of Lemma~\ref{lem:extension}] 
 Since $E$ is compact and closed, $h$ is uniformly continuous on $E$, that is, for every $\varepsilon>0$ there exists a {\em single}   $\delta(\eps)>0$ such that for any point $x\in\scrC$ the continuity condition  \bc $|y-x|<\delta(\eps) \Rightarrow |h(y)-h(x)|<\varepsilon$  \ec  is satisfied. 
 
 {\em Idea.} 
  We do not know $\delta(\eps)$ for a given $\eps$, but we can search for it using in parallel the following construction for different candidate $\delta$'s: 
 
We split the unit interval into intervals $(I_k)_k$ of length $\delta$ and write $l_k$ for the left border point of $I_k$. Write $i_k$ for $\inf\{h(x)\mid x\in I_k \cap \scrC\}$ and $s_k$ for $\sup\{h(x)\mid x\in I_k \cap \scrC\}$ if these values exist. We use lemma \ref{approx_lemma} to approximate $i_k$ and $s_k$ for all intervals, and at the same time we enumerate $E^c$.

We do this until we have found a $\delta$ (called {\em $\eps$-fit}) such that every interval has been dealt with;  by this we mean that for every interval $I_k$ we have either covered $I_k$ with $E^c$, or we have found that $s_k - i_k < \eps$, that is we already know $h$ up to precision $\eps$. In the latter case we set our approximation to $h$ to be the line from point $(l_k,i_k)$ to point $(l_{k+1},s_k)$; on the remaining intervals (the former case) we interpolate linearly. Call the new function $g_0$. We can then output $g_0(x)$ up to precision $\eps$ at any point $x \in [0,1]$.
 
{\em A problem with this construction.} The following problem can occur with $g_0$: Assume we have for some $\eps$ found a $\delta$  that is $\eps$-fit. We construct $g_0$  as described and interpolate linearly on, say, the maximal connected sequence $I_k,\dots,I_{k+n}$, all contained in $E^c$. But if we look at the same construction for $g_0$ at a better precision $\eps^\prime < \eps$, we might actually enumerate more of $E^c$ until we find an $\eps^\prime$-fit $\delta^\prime$, and this might extend the sequence $I_k,\dots,I_{k+n}$ to, say, $I_{k-i},\dots,I_{k+n+j}$, all contained in $E^c$. The linear interpolation on this sequence of intervals would then be {\em significantly flatter} than at level \eps. So for some $x \in I_{k-i}\cup\dots\cup I_{k+n+j}$ we might have that the approximation to $g_0$ with precision \eps differs by more than \eps from the approximation to $g_0$ with precision $\eps^\prime$, which is not allowed.

To fix this problem we need to define $g$ inductively over all precision levels, and \qt{commit} to all linear interpolations that have happened at earlier precision levels, as will be described now.
 
 {\em Formal construction.} Assume we want to compute the $n$-th entry of a Cauchy name for $g(x)$, that is we want to compute $g(x)$ up to precision $2^{-n}$. We say that {\em we are at precision level $n$}.  We do not know $\delta(2^{-n})$ so we do the following with $\delta=2^{-p}$ in parallel for all $p$ until we find a $\delta$ that is {\em $n$-fit}, defined as follows:

\medskip
\leftskip .5in
Split the interval $[0,1]$ into intervals of length $\delta$ and write \[I_k=[(k-1)\cdot 2^{-p},k\cdot 2^{-p})\] for the $k$-th interval and $l_k=(k-1)\cdot 2^{-p}$ for the left border point of $I_k$. For mathematical precision set $I_p:=[1-2^{-p},1]$.
% We write $s_k$ for $\sup\{h(x)\mid x\in I_k \cap \scrC\}$ and $i_k$ for $\inf\{h(x)\mid x\in I_k \cap \scrC\}$ if these values exist. 
Call an interval $I_k$ $n$-treated if there exists a smaller precision level $n^\prime < n$ where $I_k$ has been covered by $E^c$ and therefore a linear interpolation on $I_k$ has been defined. We say that $\delta$ is {\em $n$-fit} if 
	\begin{itemize}
	\item for every interval $I_k$ we have that
		\begin{enumerate}
			\item $I_k$ is $n$-treated or
			\item $I_k \cap \scrC = \emptyset$ or \label{first_condition}
			\item $s_k - i_k < 2^{-n}$\label{second_condition}
		\end{enumerate}
	\item and if for all $k$, where both $I_k$ and $I_{k+1}$ fulfill condition \ref{second_condition}, we have $i_{k+1} < s_k$; that is, we have that intervals that directly follow each other have a \qt{vertical overlap} in their approximations.
	\end{itemize}
\medskip	
	
\leftskip 0in

The following linear interpolation is an $2^{-n}$-close approximation to $g$: 

First, inductively replay the construction for all precision levels $n^\prime < n$ to find all $n$-treated intervals. For the remaining intervals, run in parallel the right-c.e.\ and left-c.e.\ approximations to $s_k$ and $i_k$, respectively, and the enumeration of $E^c$, until for every interval either condition \ref{first_condition} or \ref{second_condition} are satisfied.

Build the following piecewise linear function: For all intervals that are already $n$-treated, keep the linear interpolations from the earlier precision level $n-1$. In all remaining intervals $I_k$ that fulfill condition \ref{second_condition} we set $g$ to be the line from point $(l_k,i_k)$ to $(l_{k+1},s_k)$, that is,  for $x\in I_k$ we let  
\[g(x)= i_{k} + (x - l_k) \cdot \frac{(s_{k}-i_{k})}{l_{k+1} - l_k}= i_{k} + (x - l_k) \frac{(s_{k}-i_{k})}{\delta}.\]

Now look at the remaining intervals that have not yet been assigned a linear interpolation. In every maximal connected sequence $I_k, \dots, I_{k+n}$ of such intervals every interval must fulfill condition \ref{first_condition}. We interpolate linearly over $I_k, \dots, I_{k+n}$ in the straightforward way, that is we draw a line from point $(s_k, g(l_k))$ to point $(s_{k+n+1},g(l_{k+n+1}))$ (strictly speaking $g(l_k)$ is not yet defined, so use $\lim_{x\rightarrow l_k} g(x)$ instead).
%that is for $x\in I_k \cup \dots \cup I_{k+n}$ we set \[g_0(x)= s_{k-1} + (x - l_k) \cdot \frac{(i_{k+n+1}-s_{k-1})}{l_{k+n+1} - l_k}= s_{k-1} + (x - l_k) \cdot \frac{(i_{k+n+1}-s_{k-1})}{(n+1)\cdot\delta}.\]

{\em Verification.} It is clear that $g$ is everywhere defined and non-decreasing. Whenever $h$ was defined on a point $x\in I_k$ inside \scrC, $g$ gets assigned a value between $i_k$ and $s_k$ and since $s_k > h(x) > i_k$ and $s_k - i_k < 2^{-n}$ we have $|h(x) - g(x)| <2^{-n}$. To see that $g$ is computable note that $s_k$ and $i_k$ are defined on any interval that is not entirely contained in $E^c$ and that these values can be approximated in a right-c.e.\ and left-c.e.\ way, respectively, by using lemma \ref{approx_lemma}.
%
%{\em Constructing $g$ from $g_0$.} To solve this problem with proceed as follows: Whenever we want to compute $g(x,\eps)$ we choose the smallest $\ell$ with $2^{-\ell}\leq \eps$ and compute $g_0(x,2^{-\ell})$. In addition, we modifiy the procedure computing $g_0$ as follows: When we want to compute $g_0(x,2^{-\ell})$ we first rerun the procedure inductively for all $\eps=2^{-k+1},\dots,2^{-1}$. We commit MEANS? to all linear interpolations that have happened at those earlier precision levels.   UNCLEAR-- what is $g_0$ now? The old one? The new one?
%
%   We should have an intuitive description like the explanations above, and then one correct, formal construction (ANDRE) \qed
\end{proof}

%############### END RUPERT ##################

\begin{proof}[Proof of Lemma~\ref{lem:porous}]       % As usual, given a finite string $\sss$, we identify the cylinder $\Cyl \sss$ with the interval $$(0. \sss - \tp{-\sssl -1}, 0. \sss +\tp{-\sssl -1}).$$
%For a set $S \sub \strcantor$, let $\Cyl S = \bigcup_{\sss \in S} \Cyl \sss$.

In this proof, we say that a string $\sss$ meets $\mathcal{C}$ if $\Cyl \sss \cap \mathcal{C} \neq \ES$.

Fix $c \in \NN$ such that    $\mathcal{C}$  is porous at $z$ via $\tp{-c+2}$.
For each string $\sss$ consider the set of minimal  ``porous'' extensions at stage $t$,
$$  N_t(\sss) = \left\{ \rho \succeq \sss \,\middle|  \,
\ex \tau \succeq \sss \left[\begin{array}{c}
   |\tau| = |\rho| \lland  |0.\tau - 0. \rho | \le \tp { - |\tau|  + c} \lland  \\  \Cyl \tau \cap \mathcal{C}_t = \ES  \lland
 \rho \text{ is minimal with this property}
\end{array}\right]
\right\}. $$

%
%\ex \tau  [ |\tau| = |\rho| \lland  |0.\tau - 0. \rho | \le \tp { - |\rho|  + c} \\
%\lland \Cyl \tau \cap \mathcal{C}_t = \ES \lland \rho  \ \text{  minimal such}

We claim that
\begin{equation} \label{eqn:Ntbound} ~ ~ ~ ~ ~ ~ ~ ~ ~ ~  ~ ~ ~ ~ ~ ~  ~ ~ ~ ~ ~ ~ \sum_{\begin{subarray}{c}\rho \in N_t(\sss) \\ \rho \text{ meets } \mathcal{C}\end{subarray}} \tp{-|\rho|} \le (1- \tp{-c-2}) \tp{-\sssl}. \end{equation}
To see this, let $R$ be the set of strings $\rho$ in (\ref{eqn:Ntbound}). Let $V$ be the set of    prefix-minimal strings   that occur as witnesses $\tau$ in~(\ref{eqn:Ntbound}). Then the open sets generated by $R$ and by $V$ are disjoint. Thus, if $r$ and $v$ denote their measures, respectively, we have $r+v \le \tp{-\sssl}$.  By definition of $N_t(\sss)$, for each $\rho \in R$ there is $\tau \in V$ such that  $|0.\tau - 0. \rho | \le \tp { - |\tau|  + c}$. This implies $r \le \tp{c+1}v$. The two inequalities together imply~(\ref{eqn:Ntbound})  because $r \le \tp{c+1}(1-r)$ implies $r \le 1-1/(2^c+1) +1$.

Note that by the formal details of this definition even the \qt{holes} $\tau$ are $\rho$'s, and therefore contained in the sets $N_t(\sss)$. This will be essential for the proof of the first of the following two claims.
At each stage $t$ of the construction we define recursively a sequence of anti-chains  as follows.

 $$B_{0,t} = \{\estring\},\textnormal{ and for } n>0\colon B_{n,t} = \bigcup \{N_t(\sss)  \colon \,  \sss \in B_{n-1,t}\}$$

\noindent {\em Claim.}   If a string $\rho$  is in $B_{n,t}$ then it has a prefix $\rho'$ in $B_{n, t+1}$.

\noindent 	This is clear for $n=0$. Suppose  inductively that  it holds for $n-1$.  Suppose further that $\rho $ is in $B_{n,t}$ via a string $\sss \in B_{n-1,t}$. By the inductive hypothesis there is $\sss' \in B_{n-1, t+1}$ such that  $\sss' \preceq \sss$. Since $\rho \in N_{t}(\sss)$, $\rho$  is a viable extension of $\sss'$ at stage $t+1$ in the definition of $N_{t+1}(\sss')$,  except maybe for the  minimality. Thus there is $\rho' \preceq \rho $ in $N_{t+1}(\sss')$. \hfill $\Diamond$

\vsp

\noindent {\em Claim.}  For each $n,t$, we have
$ \sum \{ \tp{- |\rho| } \colon \, \rho \in  B _{n,t} \lland  \rho \ \text{\rm meets} \  \mathcal{C}\} \le ( 1- \tp{-c-2})^n. $

\noindent This is again clear for $n=0$.  Suppose  inductively it holds for $n-1$.  Then, by (\ref{eqn:Ntbound}),

%\[\begin{array}{rclcl}
%\displaystyle\sum_{\begin{subarray}{c} \rho \in  B _{n,t} \\ \rho \text{ meets } \mathcal{C}\end{subarray}} \tp{- |\rho| }  & = & \displaystyle\sum_{\begin{subarray}{c}  \sss \in B_{n-1,t} \\ \sss \text{ meets } \mathcal{C}\end{subarray}} \displaystyle\sum_{\begin{subarray}{c} \rho \in N_t(\sss) \\ \rho \text{ meets } \mathcal{C}\end{subarray}} \tp{-|\rho|}  \\

%& \le &  \displaystyle\sum_{\begin{subarray}{c} \sss \in B_{n-1,t} \\ \sss \text{ meets } \mathcal{C}\end{subarray}} \tp{-\sssl } ( 1- \tp{-c-2})  &  \le &  ( 1- \tp{-c-2})^n.
%\end{array}\]
$$
\displaystyle\sum_{\begin{subarray}{c} \rho \in  B _{n,t} \\ \rho \text{ meets } \mathcal{C}\end{subarray}} \tp{- |\rho| }   =  \displaystyle\sum_{\begin{subarray}{c}  \sss \in B_{n-1,t} \\ \sss \text{ meets } \mathcal{C}\end{subarray}} \displaystyle\sum_{\begin{subarray}{c} \rho \in N_t(\sss) \\ \rho \text{ meets } \mathcal{C}\end{subarray}} \tp{-|\rho|}
 \le   \displaystyle\sum_{\begin{subarray}{c} \sss \in B_{n-1,t} \\ \sss \text{ meets } \mathcal{C}\end{subarray}} \tp{-\sssl } ( 1- \tp{-c-2})    \le   ( 1- \tp{-c-2})^n.$$
This establishes the claim.\hfill $\Diamond$

\vsp

Now let $U_n = \bigcup_t \Cyl {B_{n,t}}$. Clearly the sequence $(U_n) \sN n$ is uniformly effectively open. By the first claim, $U_n  = \bigcup_t  \Cyl {B_{n,t}}$ is a nested union, so the second claim  implies that $\leb ( \mathcal{C} \cap U_n) \le ( 1- \tp{-c-2})^n$. We  show $z \in \bigcap U_n$ by induction on $n$. Clearly $z \in U_0$. If $n>0$ suppose inductively $\sss \prec z$ where $\sss \in \bigcup_t B_{n-1,t}$. Since $z$ is random there is $\eta$ such that $\sss \prec \eta \prec z$ and $\eta$ ends in $0^c1^c$.
Every interval $(a,b) \sub [0,1]$ contains an interval of the form $\Cyl \rho$ for a dyadic string $\rho$ such that  the length of $\Cyl \rho$ is no less than $(b-a) /4$. Thus, since $\mathcal{C}$  is porous at $z$ via $\tp{-c+2}$,  there is $t$, $\rho \succeq \eta$ and $\tau$ satisfying the condition in  the definition of $ N_t(\sss)$. By the choice of $\eta$ one verifies that $\tau \succeq \sss$. Thus $z \in U_n$.

Now   take  a computable subsequence $(U_{g(n)}) \sN n$ such that $\leb ( \mathcal{C} \cap U_{g(n)}) \le \tp{-n}$  to obtain     a difference test that $z$ fails.
\end{proof}

%%%%%%%%%%%%%%%%%%%%%%%%%%%%%%%%

\subsection{Sets of non-differentiability for single computable functions $f\colon \, [0,1] \to \RR$} % (fold)
\label{sub:_sets_of_non_differentiability_for_computable_functions}

% subsubsection _sets_of_non_differentiability_for_computable_functions_fcolon_,_0_1_to_rr_ (end)	

The plan is to characterize the sets $\{z\colon \, f'(z)  \, \text{is undefined}\}$ for computable functions on the unit interval. We also want to consider the case when the functions have    additional analytical  properties, such as being of  bounded variation, or being Lipschitz. 
 
 \subsubsection{The classical case}

Zahorski \cite{Zahorski:46}  showed the following (also see Fowler and Preiss \cite{Fowler.Preiss:09}). 

\begin{theorem} The sets of non-differentiability of a continuous function $f\colon \RR \to \RR$ are exactly of the form $L \cup M$, where $L$ is $G_\delta$ (ie boldface $\PI 2$) and $M$ is a $G_{\delta \sss}$ (ie boldface $\SI 3$) null set. \end{theorem}
	
\begin{proof} 
	Given $f$, the set \bc  $L= \{z \colon \ol Df(z) = \infty \land \ul Df(z)= -\infty\}$ \ec  is $G_\delta$.  Let 
	\bc $M = \{z \colon z \not \in L \land \exists p< q \, \big [ p, q \in \QQ \land  \ul Df(z)< p  \land q < \ol Df(z) \big ]\}$. \ec Then $M$ is $G_{\delta \sss} $.
	
	Note that $f'(z)$ is undefined  iff $z \in L \cup M$, and $M$ is null  
	by the Denjoy alternative.
	
	For the converse direction, we are given $L, M$ and have to build $f$. We may assume that $L, M$ are disjoint, else replace $M$ by $M-L$.  Zahorski builds a continuous function $g$ which is non-differentiable  exactly in $L$ (this is hard). Fowler and Preiss build a Lipschitz function $h$ nondifferentiable exactly in $M$. Now let $f = g+h$. 
	
	\subsubsection{The algorithmic case} If $f$ is computable,  then the set  $L$   defined above is lightface $\PI 2$, and $M$ is lightface $\SI 3$. 
	
	For the other direction,   it's not clear what happens now. Suppose $L$ is lightface $\PI 2$ and $M$ is lightface $\SI 3$.  Fowler and Preiss don't say what Zahorski did, and the original paper is probably awful to read. However,  the proof \cite[Thm 3.1]{Brattka.Miller.ea:nd} is likely to  come close. There, a $\PI 2 $ set $L$  is given, and one builds a computable function $g$ that is   nondifferentiable in $L$ and differentiable outside  a $\PI 2$ set $\widehat L \supseteq L$, where  after modifying that    construction a bit, one  can ensure that $\widehat L - L$ is null. 
	
	Since the Fowler-Preiss function for $M$  is Lipschitz, we can't hope this part of the construction  works algorithmically, because   a  computable Lipschitz function is differentiable at all computably randoms by \cite[Section 5]{Brattka.Miller.ea:nd}.
\end{proof}

\subsection{Extending  the results of \cite{Brattka.Miller.ea:nd} to higher dimensions.}
	So far, most of the interaction between randomness and computable analysis  is restricted to the 1-dimensional case. Anyone who has done calculus will confirm that differentiation becomes more interesting (and challenging) in higher dimensions. When randomness notions are defined for infinite sequences of bits, that is points in ÒCantor spaceÓ $\cantor$,   it is  irrelevant to procede to higher dimensions, because $(\cantor)^n$  is in all aspects (metric, measure) equivalent to $\cantor$ . This is clearly  no longer true when randomness for $n$-tuples of reals (yes, reals)  in the unit interval $[0,1]$ is studied, because the $n$-cube $[0,1]^n$  for $n> 1$  is much more complex than the unit interval.  
	
	What I am trying to say is that the   identification of Cantor space and $[0,1]$ doesn't always make sense any longer for higher dimensions, because methods typical for   $n$-space aren't there in the setting of  $(\cantor)^n\cong \cantor$. It makes sense in the setting of $L_p$ computability, though. Also  we can still  use the identification  to \emph{define} notions such as computable randomness in $[0,1]^n$, if pure measure theory won't do it (it does  for \ML\  and  Schnorr randomness).
	
	Currently several researchers investigate extensions  of the results in \cite{Brattka.Miller.ea:nd}  to higher dimensions. Already   Pathak~\cite{Pathak:09} showed that a  weak form of the Lebesgue   differentiation theorem  holds for \ML{} random points in the $n$-cube $[0,1]^n$. Work in progress of Rute, and Pathak, Simpson, and Rojas might  strengthen this to Schnorr random points in the $n$-cube. On the other hand, functions of bounded variation can be defined in higher dimension \cite[p.\ 378]{Bogachev.vol1:07}, and    one might try to characterize \ML{} randomness in higher dimensions via their differentiability.

By work submitted May 2012  of \cite{Freer.Kjos.ea:nd},  computable randomness can be characterized via differentiability of computable Lipschitz functions. The obvious analog of computable randomness in higher dimensions has been introduced in \cite{Brattka.Miller.ea:nd,Rute:12}. 
\begin{definition} Call a point $x =(x_1, \ldots, x_n)$ in the $n$-cube computably random if  no computable martingale succeeds on the   binary expansions of  $x_1, \ldots, x_n$ joined  in the canonical way (alternating between the sequences). \end{definition} Rute studies it via measures.  Differentiability of Lipschitz functions in higher dimensions has been studied to great depth (see for instance \cite{Alberti.Csornyei.ea:10}). Effective aspects of this theory  could be used as an approach to such a randomness notion.  One could investigate whether this is equivalent to differentiability at~$x$ of all computable Lipschitz functions.   In \cite{Freer.Kjos.ea:nd} it is shown that if $z\in [0,1]^n$ is not computably random, some computable Lipschitz function is not differentiable at $z$ (in fact some partial fails to exist). This is done by a straightforward reduction to the 1-dimensional case.

	In any computable probability space, to be weakly 2-random means to be in no null $\PI 2$ class. 
		For weak $2$-randomness in $n$-cube, new work of Nies and Turetsky yields  the analog of \cite[Thm 3.1]{Brattka.Miller.ea:nd}. The following writeup is  due to Turetsky.

\begin{theorem}[almost]
Let $z \in [0,1]^n$. Then the following are equivalent:
\begin{enumerate}
\item $z$ is weakly 2-random;
\item each a.e.\ differentiable computable function $f \colon \, [0,1]^n \to \RR$ is differentiable at $z$;
\item each a.e.\ differentiable computable function $f \colon \, [0,1]^n \to \RR$ is G\^ateaux differentiable at $z$; and
\item each a.e.\ differentiable computable function $f \colon \, [0,1]^n \to \RR$ has all partial derivatives at~$z$.
\end{enumerate}
%		$z$ is weakly 2-random $\LLR$  
%			\n   each a.e.\ differentiable computable function $f \colon \, [0,1]^n \to \RR$  is differentiable at $z$.
\end{theorem}

\begin{proof}
$(2) \Rightarrow (3) \Rightarrow (4)$ are facts from classical analysis.  We show $(1) \Rightarrow (4)$, $(1) + (4) \Rightarrow (2)$ and $(4) \Rightarrow (1)$.

\

$(1) \Rightarrow (4)$.  Suppose~$z$ is weakly 2-random and~$f$ is an a.e.\ differentiable computable function.  Fix coordinate~$i$.  For $\vec{a} \in [0,1]^n$, $h \in \RR$, let
\[
S_f^i(\vec{a}, h) = \frac{f(a_1, \dots, a_i + h, \dots, a_n) - f(a_1, \dots, a_i, \dots, a_n)}h.
\]
Recall that $\partial f/\partial x_i$ exists at~$\vec{a}$ precisely if the upper derivative $\overline{D}^if(\vec{a}) = \limsup_{|h| \to 0} S_f^i(\vec{a},h)$ and the lower derivative $\underline{D}^if(\vec{a}) = \liminf_{|h| \to 0} S_f^i(\vec{a},h)$ are finite and equal.

For~$q$ a rational, $\overline{D}^if(\vec{a}) \geq q$ is equivalent to the formula.
\[
(\forall p < q) (\forall \delta > 0) (\exists |h| < \delta) [S_f^i(\vec{a},h) > p]
\]
By density, we can take~$p$ and~$\delta$ to range over the rationals.  Since~$f$ is computable, and thus continuous, we can take~$h$ to range over the rationals.  Thus $\{ \vec{a} \mid \overline{D}^if(\vec{a}) \geq q\}$ is a $\Pi^0_2$-set uniformly in~$q$.  Symmetrically, so is $\{ \vec{a} \mid \underline{D}^if(\vec{a}) \leq q\}$.  Then the~$\vec{a}$ such that $\partial f/\partial x_i$ does not exist are precisely those~$\vec{a}$ satisfying
\[
(\forall q)[\overline{D}^if(\vec{a}) \geq q] \vee (\forall q)[\underline{D}^if(\vec{a}) \leq q] \vee (\exists q, p)[\underline{D}^if(\vec{a}) \leq q < p \leq \overline{D}^if(\vec{a})].
\]
This is a $\Sigma^0_3$-set contained in the set of all points at which~$f$ is not differentiable, and thus has measure~$0$.  Thus it cannot contain~$z$, and so $\partial f/\partial x_i$ exists at~$z$.

\

$(1) + (4) \Rightarrow (2)$.  Suppose~$z$ is weakly 2-random,~$f$ is an a.e.\ differentiable computable function and $\frac{\partial f}{\partial x_i}(z)$ exists for every~$i$.  Let~$S_f^i(\vec{a},h)$ be as defined above.

For $\vec{a} \in [0,1]^n$ and $h \in \RR$, let
\[
J_f(\vec{a}, h) = \left[\begin{array}{ccc} S_f^1(\vec{a},h) & \dots & S_f^n(\vec{a},h) \end{array}\right].
\]
By definition, $\lim_{|h| \to 0} J_f(\vec{a},h) = J_f(\vec{a})$, the Jacobian of~$f$ at~$\vec{a}$ (when this exists).

Again by definition, the derivative of~$f$ exists at~$\vec{a}$ if $J_f(\vec{a})$ exists and
\[
\lim_{||\vec{h}|| \to 0} \frac{f(\vec{a} + \vec{h}) - f(\vec{a}) - J_f(\vec{a})\vec{h}}{||\vec{h}||} = 0.
\]
By continuity of~$f$, we can take~$\vec{h}$ to have rational coordinates.

Let~$X$ be the $\Pi^0_3$-set consisting of those~$a$ satisfying
\[
(\forall \epsilon > 0) (\exists \delta > 0) (\forall ||\vec{h}|| < \delta) (\forall |b| < \delta) \frac{\left| f(\vec{a} + \vec{h}) - f(\vec{a}) - J_f(\vec{a}, b)\vec{h}\right| }{||\vec{h}||} < \epsilon.
\]
Here $\epsilon, \delta$ and~$b$ are rationals, and $\vec{h}$ has rational coordinates.  We argue first that~$X$ contains every point at which~$f$ is differentiable.

Suppose~$\vec{a}$ is a point at which~$f$ is differentiable.  Fix $\epsilon > 0$.  Let $\delta$ be sufficiently small that
\[
(\forall ||\vec{h}|| < \delta) \frac{\left|f(\vec{a} + \vec{h}) - f(\vec{a}) - J_f(\vec{a})\vec{h}\right|}{||\vec{h}||} < \epsilon/2,
\]
and also
\[
(\forall |b| < \delta) \vectornorm{ (J_f(\vec{a},b) - J_f(\vec{a}))^T } < \epsilon/2.
\]
Here we treat $J_f(\vec{a},b) - J_f(\vec{a})$ as a row vector.  Then for any $||\vec{h}|| < \delta$ and $|b| < \delta$,
\begin{eqnarray*}
\frac{ \left| f(\vec{a} + \vec{h}) - f(\vec{a}) - J_f(\vec{a}, b)\vec{h}\right| }{||\vec{h}||}
&=& \frac{\left|f(\vec{a} + \vec{h}) - f(\vec{a}) - J_f(\vec{a})\vec{h} + J_f(\vec{a})\vec{h} - J_f(\vec{a}, b)\vec{h}\right|}{||\vec{h}||}\\
&\leq& \frac{\left|f(\vec{a} + \vec{h}) - f(\vec{a}) - J_f(\vec{a})\vec{h}\right|}{||\vec{h}||}
	+ \frac{\left|(J_f(\vec{a}) - J_f(\vec{a},b))\vec{h}\right|}{||\vec{h}||}\\
&<& \epsilon/2 + \frac{\vectornorm{(J_f(\vec{a}) - J_f(\vec{a},b))^T}||\vec{h}||}{||\vec{h}||}\\
&<& \epsilon/2 + \epsilon/2 = \epsilon.
\end{eqnarray*}
Thus~$X$ contains every point at which~$f$ is differentiable, and so~$X$ has full measure.  So~$z \in X$.

Next we show that~$f$ is differentiable at~$z$.  Fix~$\epsilon > 0$.  Let~$\delta$ be sufficiently small that
\[
(\forall ||\vec{h}|| < \delta)(\forall |b| < \delta) \frac{\left|f(z + \vec{h}) - f(z) - J_f(z,b)\vec{h}\right|}{||\vec{h}||} < \epsilon/2,
\]
and fix some $|b| < \delta$ such that
\[
\vectornorm{ (J_f(z,b) - J_f(z))^T } < \epsilon/2.
\]
Then, similar to the above, for any $||\vec{h}|| < \delta$,
\begin{eqnarray*}
\frac{ \left| f(z + \vec{h}) - f(z) - J_f(z)\vec{h}\right|}{||\vec{h}||}
&=& \frac{ \left| f(z + \vec{h}) - f(z) - J_f(z,b)\vec{h} + J_f(z,b)\vec{h} - J_f(z)\vec{h}\right|}{||\vec{h}||}\\
&\leq& \frac{ \left| f(z + \vec{h}) - f(z) - J_f(z,b)\vec{h}\right|}{||\vec{h}||} + \frac{ \left| (J_f(z,b) - J_f(z))\vec{h} \right|}{||\vec{h}||}\\
&<& \epsilon/2 + \epsilon/2 = \epsilon.
\end{eqnarray*}
Thus~$f$ is differentiable at~$z$.

\

$(4) \Rightarrow (1)$.
\end{proof}

\subsection{The Lebesgue differentiation theorem} {\small   A version of the Lebesgue diferentiation theorem in one dimension can be found, for instance, in \cite[Section 5.4]{Bogachev.vol1:07}.
	\begin{thm}
 Let $g \in L_1( [0,1])$. Then $\leb$-almost surely, 

	\[g(z) = \lim_{r \ria 0}\frac 1 {r} \int_{z}^{z+r} g \, d\leb. \]
\end{thm} 
In other words, if $G(z)= \int_0^z g(x)  dx$, then for $\leb$-almost every~$z$, $G'(z)$ exists and equals $g(z)$. (This explains the name given to the theorem.)
In one dimension, the  %version given in Theorem~\ref{Thm:LebesgueDiff} above becomes
usual version becomes
\[g(z) = \lim_{r,s \ria +0}\frac 1 {r+s} \int_{z-r}^{z+s} g d\leb. \] 
This is equivalent to the formulation above that $G'(z) = g(z)$ for a.e.\ $z$. For if $f$ is a function on $[0,1]$, then \bc $f'(z) $ exists $\LR$ $ \lim_{r,s \ria 0+} S_f(z-s,z+r)$ exists, \ec in which case they are equal. 
To see this, note that if the limit on the right exists, it clearly equals $f'(z)$. Now suppose conversely that $c= f'(z)$ exists. Given $\eps>0$, choose $\delta >0$ such that for $r,s >0$, 

\bc $\max(r,s) < \delta $ implies $|c-S_f(z-s,z)| \le \eps \lland |c-S_f(z,z+r)| \le \eps $. \ec
Then {\small \begin{eqnarray*} |(s+r) c - (f(z+r) - f(z-s)) | & = & |sc - s S_f(z-s, z) + rc - rS_f(z,z+r)| \\
											& \le & |sc - s S_f(z-s, z)| + | rc - rS_f(z,z+r)| \\
											& \le & \eps (s+r) \end{eqnarray*}
			
				Hence $|c - S_f(z-s, z+r)| \le \eps$.										 
			 }

\newpage

 \section{$K$-triviality and incompressibility in computable metric spaces}

Nies and   PhD student Melnikov (co-supervised with Khoussainov) worked in Auckland and on Rakiura/Stewart Island (December 2010). 
%For questions: mail to \email{andre@cs.auckland.ac.nz}.

For more detail than this, see the    their     paper ``\texttt{http://dl.dropbox.com/u/370127/papers/MelnikovNiesKtrivMetric.pdf}{$K$-triviality in  computable metric spaces}'' \cite{Melnikov.Nies:12}, although arguments and some definitions  are a bit different in that improved version.  Also see Nies' Talk ``\texttt{http://dl.dropbox.com/u/370127/talks/MelnikovNiesKtrivial.pdf}{K -triviality in computable metric spaces}'' at the Cape Town 2011 CCR. (Both are available on Nies' web page.)

\subsection{Background on computable metric spaces}

 \begin{definition}
	{\rm Let $(M, d)$ be a Polish space, and let $(q_i)\sN i$ be a dense sequence in $M$ without repetitions. We say that $\mathcal M = (M, d, (q_i)\sN i )$ is a \emph{computable metric space}    if $ d(q_i, q_k)$ is a computable real uniformly in $i,k$. We say that $(q_i)\sN i$ a \emph{computable structure} on $(M,d)$, and refer to the elements of the sequence $(q_i)\sN i$ as the \emph{special points}.} 
\end{definition}
\begin{definition}
	{\rm A sequence $(p_s) \sN s$ of special points is called a \emph{Cauchy name} if $d(p_s, p_{t}) \le \tp{-s}$ for each $s\in \NN, t \ge s$. Since the metric space is complete, $x= \lim_s p_s$ exists; we say that $(p_s) \sN s$ is a Cauchy name for $x$. Note that $d(x,p_s) \le \tp {-s}$ for each~$s$. } 
\end{definition}
\begin{definition} \label{def:CompPoint}
	{\rm We say that a point in a computable metric space is \emph{computable} if it has a computable Cauchy name. } 
\end{definition}

If  a computable metric space $\+ M = (M, d, (q_i)\sN i)$ is fixed in the background,  we will  use the letters $p,q$ etc.\ for special points. We may identify the special point $q_i$ with $i \in \NN$. Thus, we may view a Cauchy name as a function $\aaa \colon \, \NN \to \NN$. We also write $\lim_n \aaa(n) = x$ meaning that $\lim_n q_{\aaa(n)} = x$.
\begin{example}
	\label{exam:MetricExamples} {\rm The following are computable metric spaces.
	
\begin{itemize}  \item[(i)] 	The unit interval $[0,1]$ with the usual distance function  and the computable structure given by some effective listing without repetition of the rationals in $[0,1]$, i.e., by  fixing a computable bijection $\tau$ between $\NN$ and $\QQ \cap [0,1]$, and letting $q_n = \tau(n)$. 
	
	\item[(ii)] The Baire space $\NN^\NN$ consisting of the functions $f \colon \NN \to \NN$ with the usual ultrametric distance function $$d(f,g) = \max \{ \tp{-n} \colon \, f(n) \neq g(n)\},$$ (where $\max \ES = 0$).   The computable structure is given by fixing some effective listing without repetition  of the functions that are eventually $0$.  (Note that such functions can be described by strings in $\strbaire$ that   don't end in $0$.)
	 
 \item[(iii)] Cantor space $\cantor \sub \NN^\NN$, with the inherited distance function and computable structure.
\end{itemize}}
\end{example}

 \subsection{$K$-triviality}

We will generalize the usual definition of $K$-triviality  in Cantor space
\bc $\ex b \, \fa n \, K(x\uhr n) \le K(0^n)+b$  \ec to   points in general computable metric spaces. 

We first provide some  preliminary material. Thereafter, we introduce our main concept.

\subsubsection*{$K$-triviality for functions} Fix some effective encoding of tuples $x$ over $\NN$ by binary strings, so that $K(x)$ is defined for any such tuple. 

\begin{definition} \label{def:Ktriv_function}  We say that a function $\aaa \colon \, \NN \to \NN$ is \emph{$K$-trivial} if \bc $\exo b \fao n K(\aaa \uhr n) \le K(0^n)+ b$. \ec \end{definition}
\begin{prop}
	\label{prop:graph} A function $\aaa$ is $K$-trivial iff its graph $ \Gamma_\aaa \leftrightharpoons \{\la n, \aaa(n) \ra \colon \, n \in \NN\}$ is $K$-trivial in the usual sense of sets. 
\end{prop}
\begin{proof}
One uses    that $K$-triviality implies lowness for $K$ (see \cite[Section 5.4]{Nies:book}).  
\end{proof}

%%%%%%%%%%%%%%%%%%%%%%%%%%%%
\subsubsection*{Solovay functions} \mbox{}

\n Recall that a computable function $h \colon \, \NN \to \NN$ is called a \emph{Solovay function}  \cite{Bienvenu.Downey:09} if $\fa r \, [K(r) \le^+ h(r)]$, and $\ex^\infty r \, [ K(r) =^+ h(r)]$. 

Solovay \cite{Solovay:75} constructed an example of such a function. The following simpler recent example is due to Merkle. We include the short proof for completeness' sake.
\begin{fact}
	\label{fact:Sol1} There is a Solovay function $h$. 
\end{fact}
\begin{proof}
	Let $\UM$ denote the optimal prefix-free machine. Given $r = \la \sss, n,t \ra$, if $t$ is least such that $\UM_t(\sss) = n$, define $h(r) = \sssl$. Otherwise let $h(r) = r$. 

	We have $K(r) \le^+ h(r)$ because there is a prefix-free machine $M$ which on input $\sss$ outputs $r= \la \sss, \UM(\sss), t \ra$ if $t$ is least such that $\UM_t(\sss)$ halts. If $\sss$ is also a shortest string such that $\UM(\sss) =n$, then we have $h(r) = \sssl = K(N) \le^+ K(r)$. 
\end{proof}

The following generalizes the corresponding fact for sets also due to Merkle. 
\begin{fact}
	\label{fact:Sol2} Let $\aaa \colon \, \NN \to \NN$ be a function such that $\fa r \, K(\aaa\uhr r) \le h(r) +b$. Then $\aaa$ is $K$-trivial via a constant $b+ O(1)$. 
\end{fact}
\begin{proof}
	Given $n$, let $\sss$ be a shortest $\UM$-description of $n$, and let $t$ be least such that $ \UM_t(\sss) = n$. Let $r = \la \sss, n , t \ra$. Then 

	\bc $K(\aaa \uhr n) \le^+ K(\aaa \uhr r) \le h(r)+ b = \sssl + b = K(n) +b$. \ec 
\end{proof}

\subsection{$K$-trivial points in computable metric space} Our principal notion is $K$-triviality in computable metric spaces. Recall from Definition~\ref{def:CompPoint}  that a point $x$ in a computable metric space $\+M$ is called computable if it has a computable Cauchy name. We define $K$-triviality in a similar fashion. 
\begin{definition}
	\label{def:Ktrivpoint} {\rm We say that a point $x$ in a computable metric space is \emph{$K$-trivial} if it has a Cauchy name that is $K$-trivial as a function. }
\end{definition}
\begin{prop} \label{prop:K_triv_preservation}
	Let $\+ M, \+N $ be computable metric spaces, and let the map $F\colon \, \+ M \to \+ N$ be computable. If $x$ is $K$-trivial in $\+M$, then $F(x) $ is $K$-trivial in $\+N$. 
\end{prop}
\begin{proof}
	Let $\alpha$ be a $K$-trivial Cauchy name for $x$. Since $F$ is computable, there is a Cauchy name $\beta \le_T \aaa $ for $F(x)$. Then $\beta$ is $K$-trivial by Proposition~\ref{prop:graph} and the result of \cite{Nies:AM} that $K$-triviality for sets is closed downward under~$\leT$. 
\end{proof}

If $F$ is effectively uniformly continuous (say Lipschitz), then  one can also give a direct proof which avoids the hard result from \cite{Nies:AM}. Moreover, from  a $K$-triviality constant for  $x$  one can  effectively  obtain a $K$-triviality constant for $F(x)$, which is not true in the general case.  

%To see this, note that, by Proposition~\ref{prop:euc_wtt},  we have $\beta = \Phi^\aaa$ for a Turing functional $\Phi$ with use bounded by  some computable function $g$. Then \bc $K(\beta \uhr n) \lep K(\aaa \uhr {g(n)}) \lep K(g(n)) \lep K(n)$. \ec 
%The increase in constants is fixed, because it only  depends on $\Phi$ and $g$. 

%ADAPT THIS Suppose $f$ is a $K$-trivial Cauchy name of $x \in M$. 
%
%Let $u$ be a computable function as in (\ref{eqn:inv1}) and let $g(n) = u(f(n+1), n+1)$. We have \bc $d(\beta_{u(f(n+1),n+1)} , q_{h(n)} ) \le 2^{-n-1}$, \ec and thus \bc $d(\beta_{g(n)}, x) \leq d(\beta_{g(n)}, q_{f(n+1)} )+ d(q_{f(n+1)}, x) \leq \tp{-n-1}+ \tp{-n-1} = \tp{-n}$. \ec Thus, $\{\beta_{g(n)}\}$ is a Cauchy name of $x$.
%
%We have $K(g \uhr {n}) \lep K(f \uhr {n+1}) \lep K(n+1) =^+ K(n) $.

Since the identity map is Lipschitz, we obtain: 
\begin{cor}
	$K$-triviality in a computable metric space is invariant under changing of the computable structure to an equivalent one. More specifically, if $x$ is $K$-trivial for $b$ with respect to   a computable structure, then $x$ is $K$-trivial for $b+O(1)$ with respect to an equivalent structure. 
\end{cor}

\subsection{Existence of non-computable $K$-trivial points}
\begin{example}
	There is a computable metric space $\+M$ with a noncomputable point such that the only $K$-trivial points are the computable points. 
\end{example}
\begin{proof}
	Let  $\+ M$ be the computable metric space with domain $ \{\Omega_s \colon \, s \in \NN\} \cup \{\Omega\}$,  the metric inherited from the unit interval and with the computable structure given by $q_s = \Om_s$. 

	Suppose  $g$ is a Cauchy name for $\Omega$  in $\+M$. Then, for a rational $p$, we have $p> \Omega \lra \ex n [ g(n)+ \tp{-n} <  p]$. Thus   $\{p \in \QQ  \colon p <  \Omega\}  \leT g$. This set is Turing complete, so $g$ is not $K$-trivial. 
\end{proof}

A metric space is said to be \emph{perfect} if it has no isolated points. In the following we take Cantor space $\cantor$ as the computable metric space with the usual computable structure of Example~\ref{exam:MetricExamples}.
\begin{proposition}
	 \cite[Prop.\ 6.2]{Brattka.Gherardi:09}] \label{prop:BrattGh} Suppose the computable metric space $\+M$ is perfect. Then there is a computable injective map $F\colon \, \cantor \to \+M$ which is Lipschitz with constant 1. 
\end{proposition}

Given $\delta > 0$, we denote by $B(x,\delta)$ the open ball of radius $\delta $ and with center in $x$, that is, $B(x,\delta) = \{y \in M: \, d(x,y)<\delta\}$.

\begin{theorem}
	\label{theo:K-triv} Let $\+ M$ be a computable   perfect Polish space. Then for every special point $p \in M$ and every rational $\delta >0$ there exists a non-computable $K$-trivial point $x \in B(p, \epsilon)$. 
\end{theorem}
\begin{proof}  Note that the closure in $\+ M$  of $B(p, \delta)$ is again a     perfect Polish space $   \+  N$ which has an    inherited computable structure. By the result of Brattka and Gherardi  there is a computable injective Lipschitz  map $F\colon \, \cantor \to {\+ N}$. 

Let $A $ be a non-computable $K$-trivial point in Cantor space. Then $x= F(A)$ is $K$-trivial in $\+ N$, and hence in $\+ M$, by Proposition~\ref{prop:K_triv_preservation};  actually only the easier case of Lipschitz functions discussed after~\ref{prop:K_triv_preservation} is needed. 

As  Brattka and Gherardi point out before Prop.\ 6.2, the inverse of $F$ is computable (on its domain). Thus, if  $x$ is computable then so is $A$, which is not the case.
\end{proof}

\subsection{An equivalent local definition of $K$-triviality} We fix a computable metric space $\+ M = (M, d, (q_n)\sN n)$. When it comes to $K$-triviality for points in $\+ M$, one's first thought would be to directly adapt the definition of $K$-triviality of some point~$x$ in Cantor Space,   \bc $\ex b \, \fa n \, K(x\uhr n) \le K(0^n)+b$. \ec  Recall from Example \ref{exam:MetricExamples} that in Cantor space $\cantor$ we chose as the special points the sequences of bits that are eventually $0$. Since we identify a special point $p= \aaa(n)$ with $n$, we get the following tentative definition: 
\begin{equation}
	\label{eqn:weakK} \ex b \, \fa n \, \ex p [ d(x,p ) \le \tp{-n} \, \land \, K(p) \le K(n) +b]. 
\end{equation}

This isn't right, though:
 \begin{proposition}\label{pro:weakKtrivCounterexample}
 	There is a Turing complete $\PI 1$ set $A \in \cantor$ satisfying condition~(\ref{eqn:weakK})
 \end{proposition}

\begin{proof} 
	For a string $\aaa$, let $g(\aaa)$ be the longest prefix of $\aaa$ that ends in $1$, and $g(\aaa)= \estring$ if there is no such prefix.
   We say that a set $A$ is \emph{weakly $K$-trivial} if \bc $\fao n [ K(g(A\uhr n))\lep K(n)]$. \ec  This is equivalent to (\ref{eqn:weakK}).  
(Clearly, every $K$-trivial set is weakly $K$-trivial. Every \emph{c.e.}\ weakly $K$-trivial  set is already $K$-trivial. )

We now build a   Turing complete $\PI 1$  set $A$ that is weakly $K$-trivial.  We maintain the condition  that 

		\begin{equation}
			\label{eqn:gamma_K_large} \fao i \fao w [\gamma_i < w \ria K(w)> i],
		\end{equation}
	where $\gamma_i$ is the $i$-th element of $A$. This implies that $A$ is Turing complete, as follows.   We build a prefix-free machine~$N$. When $i$ enters $\Halt$ at stage $s$, we declare that $N(0^i1)=s$. This implies $K(s)\le i+d$ for some fixed coding constant~$d$. Now $i \in \Halt \lra i \in \ES'_{\gamma_{i+d+1}}$, which implies $\Halt \leT A$. 

We let $A = \bigcap A_s$, where $A_s$ is a cofinite set effectively computed from $s$, $A_0 = \NN$, $[s, \infty) \sub A_s$, and $A_{s+1} \sub  A_s$ for each $s$.  We view $\gamma_i$ as a movable marker; $\gamma^s_i$ denotes its position at stage $s$, which is the $i$-th element of $A_s$.

\vsp

\n \emph{Construction of $A$ and a prefix-free machine $M$.}

\n \emph{Stage $0$.} Let $A_0= \NN$.

\vsps

\n \emph{Stage $s>0$.}  Suppose that there is $w$ such that $i:= K_s(w)< K_{s-1}(w)$. By convention, $w$ is unique and $w<s$. Thus, there is a new computation $\UM_s(\sss)= w$ with $\sssl = i$ at stage $s$.

If $w\le  \gamma^{s-1}_i$ then let $A_s= A_{s-1}$. If $w> \gamma^{s-1}_i$ then, to maintain  (\ref{eqn:gamma_K_large}) at stage $s$, we \emph{move} the marker $\gamma_i$: we let $A_s= A_{s-1} - [\gamma^{s-1}_i,s)$, which results in $\gamma^s_{i+k}= s+k$ for $k \ge i$, while $\gamma^s_j= \gamma^{s-1}_j$ for $j<i$.  

In any case, declare $M(\sss)= g(A_s\uhr w)$.

\vsp

\verif Clearly, each marker $\gamma_i$ moves at most $\tp{i+1}$ times, so $A = \bigcap_s A_s $ is an infinite co-c.e.\ set. Furthermore, condition (\ref{eqn:gamma_K_large}) holds because it  is maintained at each stage of the construction.

We show by induction on $s$ that 
\begin{equation}
	 \label{eqn:maintain_weak_K_triv} \fao n [K(g(A_s \uhr n)) \lep K_s(n))].
\end{equation}

For $s=0 $ the condition is vacuous. Now suppose $s>0$. We may suppose that  $w$ as in  stage $s$ of the construction exists, otherwise (\ref{eqn:maintain_weak_K_triv}) holds at stage $s$ by inductive hypothesis.

As in the construction let $i = K_s(w)$, and let $\sss$ be the string of length $i$ such that $\UM_s(\sss)= w$.  If $w \le \gamma^s_i$ then $A_s= A_{s-1}$, so setting $M(\sss) = g(A_s\uhr w) $ maintains (\ref{eqn:maintain_weak_K_triv}).

Now suppose that  $w > \gamma^s_i$. Let $n<s$.  We  verify (\ref{eqn:maintain_weak_K_triv}) at stage $s$ for $n$. 

If $n \le \gamma^{s-1}_i$ then $A_s\uhr n= A_{s-1}\uhr n$ and $K_s(n)= K_{s-1}(n)$, so the condition holds  at stage $s$ for $n$ by inductive hypothesis.  Now suppose that  $n > \gamma^{s-1}_i$. By (\ref{eqn:gamma_K_large}) at stage $s-1$ we have $K_{s-1}(n) > i$, and hence $K_s(n) \ge i$ (equality holds if $n=w$). Because we move the marker $\gamma_i$
 at stage $s$,  we have $g(A_s\uhr n) = g(A_s\uhr w)$. Thus setting $M(\sss) = g(A_s\uhr w) $  ensures that  the condition  (\ref{eqn:maintain_weak_K_triv}) holds at stage $s$ for $n$. 
\end{proof}

\subsection{Local definition of $K$-triviality} % (fold)
\label{sub:local_definition_of_k_triviality}

% subsection local_definition_of_k_triviality (end)

The analog of the definition in Cantor space $\ex b \, \fa n \, K(x\uhr n) \le K(0^n)+b$ should be a stronger property: From the string $x \uhr n$ we can compute the maximum distance $\tp{-n}$ we want the highly compressible special point $p$ to have from $x$. Thus, we should actually require that $K(p,n) \le K(n)+b$ (where $K(p,n) $ is the complexity of the pair $\la p,n\ra$). 
\begin{definition}
	\label{def:locallyK-triv} {\rm Let $b \in \NN$. We say that $x \in M$ is \emph{locally $K$-trivial via} $b$, or locally $K$-trivial($b$) for short, if  
	\begin{equation}
		\label{eqn:locK} \fa n \, \ex p \, [ d(x,p ) <  \tp{-n} \, \land \, K(p,n) \le K(n) +b]. 
	\end{equation}
	} 
\end{definition}

$K$-trivial implies locally $K$-trivial: 
\begin{fact}
	\label{fact:Cauchy nametolocKtrivial} Suppose $f$ is a $K$-trivial via $v$ Cauchy name for~$x$. Then $x$ is locally $K$-trivial via $v + O(1)$. 
\end{fact}
\begin{proof}
	Note that $d(f(n), n) \le \tp{-n}$ for each $n$. Clearly, 

	\bc $K(f(n), n) \lep K(f \uhr {n+1} ) +O(1) \le K(n+1) + v + O(1) \le K(n)+ v + O(1)$. \ec Hence, for each $n$, the point $p = f(n)$ is a special point at a distance of at most $\tp{-n}$ from $x$ such that $K(p,n) \le K(n) +v + O(1)$. 
\end{proof}

\n \emph{Showing the definition is natural.} We show that  the notion of  locally $K$-trivial point  in (\ref{eqn:locK})  is actually  independent   of the fact that we chose the bounds  on distances to be  of the form $\tp{-n}$.   We list the  positive  rationals as $(r_i)\sN i$ in an effective  way without repetitions.  Note that for $\epsilon = r_i$ we have $K(\epsilon) = K(i)$ be definition. Given that, 
we could also define local $K$-triviality($b$) like this:

\begin{equation} \label{eqn:stronger_local_Ktriv}
	 \fa \epsilon \, \ex p \, [ d(x,p ) <  \epsilon \land \, K(p,\epsilon ) \le K(\epsilon) +b ],
\end{equation}
where $\epsilon$ ranges over positive rationals. This is apparently stronger than the definition above. 
However,  if we encode positive  rationals as $(r_i)\sN i$ in a reasonable way then $\tp{-i} \le r_i$ for all  $i\ge $ some $i_0$. In this case, Fact \ref{fact:Cauchy nametolocKtrivial} still holds: if $n$ is least such that $\tp{-n} \le r_i < \tp{-n+1}$, then for  $r_i=\epsilon$ we can take $p= f(n)$ as a witness for local $K$-triviality in the sense of (\ref{eqn:stronger_local_Ktriv}): we have  $i \ge n$, and hence 

\bc $K(f(n), i) \le K(f \uhr {i+1} ) +O(1) \le K(i)+ v + O(1)$.    \ec

In Theorem~\ref{thm:equivLoc} below we will show the converse of the foregoing Fact~\ref{fact:Cauchy nametolocKtrivial}. Thus, being locally  $K$-trivial, either in the original  or the strong sense,  is all  equivalent to   having a $K$-trivial Cauchy name.

However, this equivalence holds only for points themselves, not for their approximating sequences. The following   example shows that a point~$x$ can have continuum many Cauchy names $(p_n) \sN n$ of special points so that (\ref{eqn:locK}) in Definition~\ref{def:locallyK-triv} holds for each $n$ via $p_n$. 
\begin{example}
	\label{exam:M-treespace} {\rm \label{ex:simpleM} Let the special points of the computable metric space $\+M$ be all the pairs $\la r, n\ra$, where  $r \in \{0,1\}$ and  $n \in \NN$. We declare 
	\bc $d(\la 0,n \ra, \la 1,n \ra) = \tp{-n}$ and $d(\la r, n ), \la k, n+1\ra) = \tp{-n-1}$. \ec 
	All the special points are isolated.  There is only one non-isolated  point, namely $x = \lim _n \la 0, n \ra$. This point  $x$ is computable. Each sequence $(p_n) \sN n$ of the form $(\la r_n, n \ra)\sN n$ is a Cauchy name for $x$. For an appropriate $b$,  (\ref{eqn:locK}) holds for each $n$ via $p_n$.} 
\end{example}
\begin{fact}
	\label{fact:count} $\# \{ p \in \NN \colon \, K(p,n) \le K(n) + b \} = O(2^b)$. 
\end{fact}
\begin{proof}
	This is like \cite[2.2.26]{Nies:book}, with the change that we let $M(\sss)=n$ if $\UM(\sss) = \la i,n\ra$. 
\end{proof}

We next derive a fact on the number and distribution of the points that are locally $K$-trivial($b$). Suppose that distinct points $x_1, \ldots, x_k \in M$ are locally $K$-trivial($b$). Pick $n^* \ge 2 $ so large that $ 2^{-n^*+1 } < d(x_i,x_k)$ for any $i \neq j$, and choose $p_i$ for $x_i, n^*$ according to (\ref{eqn:locK}). Then all the $p_i$ are distinct. By Fact~\ref{fact:count}, this implies that $k = O(2^b)$. Furthermore, $x_i$ is the only locally $K$-trivial($b$) point $x$ such that $d(x,p_i) \le \tp{-n^*}$. 

We summarize the preceding remarks.
\begin{lemma}
	\label{lem:O(2^b)many} {\rm \mbox{} 

	\bi 
	\item[(a)] For each $b \in \NN$, at most $O(2^b)$ many $x \in M$ are locally $K$-trivial($b$). 

	\item[(b)] There is $n^*\in \NN$, $n^* \ge 2$  as follows: for each point $x$ that is locally $K$-trivial via~$b$, there is a special point $\wt p$ with $K(\wt p,n^*) \le K(n^*) +b$ such that $x$ is the only locally $K$-trivial($b$) point within a distance of $ \tp{-n^*}$ of~$\wt p$. \ei } 
\end{lemma}

We are now ready to prove the converse of Fact~\ref{fact:Cauchy nametolocKtrivial}: local $K$-triviality in the sense of Definition~\ref{def:locallyK-triv} implies $K$-triviality in the sense of Definition~\ref{def:Ktrivpoint}. 

%%%%%%%%%%%%%%%%%%%%Big Thm!!
\begin{theorem}
	\label{thm:equivLoc} Suppose that $x\in M$ is locally $K$-trivial via $b$. Then $x$ has a Cauchy name~$h$ that is $K$-trivial via $b+ O(1)$. 
\end{theorem}
\begin{proof}
	After adding a constant to the Solovay function $h$ from Fact~\ref{fact:Sol1} if necessary, we may suppose that $\fa r \, K(r) \le h(r)$ and $\ex^\infty r \, K(r) = h(r)$. Thus,  the inequalities and equalities hold literally, not merely  up to a constant. 

	\vsp

\n {\it The c.e.\ tree $T$.} 	Choose a number $n^* $ and a special point $\wt p$ for the given point $x$ and the constant $b$ according to Lemma~\ref{lem:O(2^b)many}. 
	 We define a c.e.\ tree $T \sub \NN^{< \omega}$.  Since special points are identified with natural numbers, we can think of 	the nodes of the tree $T$  as being  labelled by special points.  The root is labelled by $\wt p$. 
	 Of course, the same label $p$ may be used at many nodes.
	We think of a node at level $n$ labelled $p$ as the first $n+1$ items of a Cauchy name.

	As usual, $[B]$ denotes the set of (infinite) paths of a tree $B\sub \NN^{< \omega}$. Each $\aaa  \in [T]$ will be a Cauchy name for some point $y$. The special points $\aaa(i)$ will be  witnesses for the local $K$-triviality($b$) of the point $y$ at $n$ , where $n= n^* +i$. 

	Formally, we let $T= \bigcup_s T_s$, where 
	\begin{eqnarray}
		\label{eqn:Tsdefinition} T_s & = & \{(p_{n^*}, \ldots, p_v)\colon \, p_{n^*} =  \wt p \, \land \\
		&& \ \ \ \fa i. n^* \le i < v \, [ K_s(p_i, i ) \le h(i)+ b \, \land \, d(p_i, p_{i+1})\le  \tp{-i-1}]\} \nonumber 
	\end{eqnarray}

	Note that $[T] $ contains a Cauchy name for $x$ by the hypothesis that $x$ is locally $K$-trivial($b$),   the choice of $\wt p$ and since $n^* \ge 2$. Actually,  any path of $T$ does this:
	\begin{claim}
		\label{claim: allpathCauchy name} Each $\aaa\in [T]$ is a Cauchy name of $x$. 
	\end{claim}
	To see this, let $p_0, \ldots, p_{n^*-1}$ witness local $K$-triviality($b$) of $x$ for $n< n^*$. Since the metric space $(M,d)$ is complete,  $\aaa$ is a Cauchy name of some point $y \in M$, and the special points \bc $p_0, \ldots, p_{n^*-1}, \aaa(n^*), \aaa(n^* +1) , \ldots$   \ec show that $y$ is locally $K$-trivial($b$). Also $d(y,\wt p) \le \tp{-n^*}$. Thus in fact $y=x$ by Lemma~\ref{lem:O(2^b)many}. This proves the claim.

		For instance, in the case of  the computable metric space of Example~\ref{exam:M-treespace}, once again  let  $x$ be  the only limit point.  We may  let $n^* = 2$  and $\wt p = \la 0,2 \ra$. Then $T$ is a full binary tree, consisting  of all the tuples of the form   $(\wt p, \la r_3 , 3 \ra, \ldots, \la r_v, v \ra)$ where $r_i \in \{0,1\}$.    

	\vsps

\n {\it  A very thin c.e.\ subtree  $G$ of $T$.}   The tree $T$  for the computable metric space in  Example~\ref{exam:M-treespace} shows that there may be lots of Cauchy names with witnesses for local $K$-triviality; we cannot expect that each such Cauchy name is $K$-trivial.  We will prune  $T$ to a c.e.\ subtree  $G$ that is so thin that all of its nodes  $\tau$ are strongly compressible in the sense that $K(\tau) \le h(|\tau|) + b + O(1)$; hence each infinite path is $K$-trivial by Fact~\ref{fact:Sol2}.  

	We say that a label $p\in \NN$ is \emph{present at level} $n$ of a tree $B\sub \NN^{< \omega} $ if there is $ \eta \in B$ such that $\eta$ has length $n$ and ends in~$p$. 	While $G$ is only a thin subtree of $T$, we will ensure that each label present at  a level $n$ of $T$ is also present at level $n$ of $G$.  This will show that $G$ still contains a Cauchy name for~$x$.

	To continue  Example~\ref{exam:M-treespace}, there are only two labels at each level of $T$, so for  $G \sub T $ we can  simply take  the tuples where each $r_i$ is $0$, except possibly the last.  Then the only infinite path of  $G$ is a computable Cauchy name of the limit point~$x$.

 We will build a computable enumeration $(G_s)\sN s$ of the tree $G$ where $G_s \sub T_s$ for each~$s$. To help with the definition of this computable enumeration, we first  define a slower computable enumeration $(\wt T_s)\sN s$ of $T$ that grows ``one leaf at a time''.  The $\widetilde  T_s$ are       subtrees of the $T_s$.

 \vsps

  {\it Why can each node  of $G$   be compressed?}
	Suppose a new leaf labelled $p$ appears at level $n$ of $T$, but is not yet at level $n$ of $G$. Suppose also that  $p$ is a successor on $T$ of a node labelled $q$. Inductively,  $q$ is already on level $n-1$ of $G$; that is, there already is a node  $\ol \eta$ of length $n-1$  on $G$ that ends in $q$.  Since $p$ is on level $n$ of $T$, there is a  $\UM$-\emph{description} showing that $K(p, n) \le h(n) +b$ (that is, a string $|w|$ with $|w| \le h(n) + b $ such that $\UM(\w) = \la p,n\ra$).  Since $p$ is not at level $n$ of $G$,  this $\UM$-description is  ``unused''. Hence we can use it as  a description of a  new node $\eta = \ol \eta \ape p$ on $G$.  This ensures that $K(\eta) \le h(n) + b + O(1)$.

		 Note that we make use of the fact that a node $\eta$ on $G$, once strongly compressible at a stage, remains so at later stages. This is why we need the Solovay function~$h$.  If we tried to satisfy  the condition $K(\eta) \le K(n) + b + O(1)$, we might fail, because $K(n)$ on the right side could decrease later on. We also needed the Solovay function to ensure that $T$ is c.e.

		  In the formal construction, we build a prefix-free machine $L$ (see \cite[Ch.\ 2]{Nies:book})  to give short descriptions of these nodes. The argument above is implemented via maintaining the conditions (\ref{eqn: Lcompr},\ref{eqn: Lprop}) below.

	\vsp

	\n {\it A  slower  computable enumeration $(\wt T_s)\sN s$ of $T$. } 	Let $\wt T_0$ consist only of the empty string. If $s>0$ and $\wt T_{s-1}$ has been defined, see whether there is $\tau \in T_s - \wt T_{s-1}$. If so, choose $\tau$ least in some effective numbering of $\NN^{< \omega}$. Pick $v$ maximal such that $\tau \uhr v \in \wt T_{s-1}$, and put $\tau \uhr {v+1} $ into $\wt T_s$. Clearly we have $T=\bigcup_s \wt T_s $.

		\vsp

	\n {\it Three conditions that need to be maintained at each stage.} 	For strings $\tau, \eta \in \strbaire$ we write $\tau \sim \eta$ if they have the same length and end in the same elements.  Recall that   each label present at  a level $n$ of $T$ needs to be also  present at level $n$ of $G$. Actually, in the construction  
	 we   ensure that for each stage $s$,  each label $p$ that is present at a level $n$ of $\wt T_s$ is also present at level $n$ of $G_s$:
		\begin{equation}
			\label{eqn: rhosimeta} \fa \tau \in \wt T_s \, \ex \eta \in G_s \, [\tau \sim \eta]. 
		\end{equation}

		To make sure that each $\eta \in G$ satisfies $K(\eta) \le h(|\eta|) + b + O(1)$, we construct, along with $(G_s)\sN s$, a computable enumeration $(L_s)\sN s$ of (the graph of) a prefix-free machine $L$. Let $m, n$ range over natural  numbers, and $v,w$ over strings. We    maintain at each stage~$s$ the conditions 
		\begin{equation}
			\label{eqn: Lcompr} \fa \eta \in G_s \, \fa m \, \big  [ 0 < m \le |\eta| \to  \ex v \   [ \, |v| \le h(m) + b \, \land \, L_s(v) = \eta \uhr m  ] \big ]; 
		\end{equation}
		\begin{equation}
			\label{eqn: Lprop} \text{if} \ \UM_s(w) = \la p, n \ra, \, \text{then} \,\big [ w\in \dom (L_s) \ \rightarrow \ p \ \text{is at level} \ n \ \text{of} \  G_s \big ]. 
		\end{equation}

	For full construction and verification see the paper. 
		\end{proof}

Note that all this also works for plain descriptive complexity  $C$ instead of $K$.   We say that $x \in M$ is \emph{locally $C$-trivial via} $b$  if  
	\begin{equation}
		\label{eqn:locC} \fa n \, \ex p \, [ d(x,p ) <  \tp{-n} \, \land \, C(p,n) \le C(n) +b]. 
	\end{equation}
	
Then by the method in the proof above but for plain machines, we have a local, ``AIT'' characterization of  computable points in computable metric spaces: 

\begin{thm} A point is locally $C$-trivial iff it is computable. \end{thm} For details see the upcoming paper.

\subsection{Incompressibility and randomness}

This material is not in the paper \cite{Melnikov.Nies:12}.

\ML~ randomness  is a central randomness notion. The usual definition (see \cite[3.2.1]{Nies:book}) via passing all \ML~ tests works in all computable probability spaces \cite{}, In particular, it can be applied both in Cantor space with the product measure, and in the unit interval with the usual uniform measure. Note that a real $r \in [0,1]$ is \ML~ random iff its binary expansion is  when viewed as an element of $\cantor$.

  The Schnorr-Levin Theorem (see \cite[3.2.9]{Nies:book}) states that     a set $A\in \cantor$ is \ML~ random if and only if there is a constant~$b\in \NN$ such that $K(A \uhr n) >  n-b$ for all $n$.

We let   \begin{eqnarray*} K(z\, ;n) & = & \min \{ K(p,n) \colon \, d(z,p) <  \tp{-n}\}, \\ 
K_*(z\, ;n)& =&  \min \{ K(p) \colon \, d(z,p) <  \tp{-n}\}. \end{eqnarray*} 
 
\begin{definition} \label{def:incompr_approx}
{\rm   We say that a point $z  $ in  a   computable metric space  $\+ M$ is \emph{incompressible in approximation (i.a.)}   if 
$\ex b \in \NN  \, \fao n \,  [ K(z; \, n) >  n-b]$. 
We say that $z$ is \emph{strongly i.a. }  if
$\ex b \in \NN  \, \fao n \,  [ K_*(z; \, n) >  n-b]$. }
\end{definition}

By the following fact, the second  notion above expresses that  the closer a special point is to $z$, the less it can be compressed. 

\begin{fact}  \label{fact:char_str_i.a.}
	$z$ is strongly i.a.\ via $b$   $\LR $  
 $\fao p \,  [ d(z, p) \ge \tp{-K(p) -b}]$.

\end{fact} 
\begin{proof} Note that the condition $\fao p \,  [ d(z, p) \ge \tp{-K(p) -b}]$  is equivalent to 
 \bc $\fao n \, [d(z,p) < \tp{-n} \ria  K(p)+b > n]$, \ec  i.e.,   $  \fa n \, K_*(z \, ;n) >  n-b$.    \end{proof}

\begin{fact}
	Let $\+ M $  be  either  Cantor space or the unit interval, with the   computable structures introduced in Example~\ref{exam:MetricExamples}. Let $x$ be an element of  $\+ M$. Then the following are equivalent.
	\begin{itemize}
		\item[(i)]  $x$ is incompressible in approximation.
		
		\item[(ii)] $x$ is strongly incompressible in approximation.  

		\item[(iii)]  $x$ is \ML\ random. 
	\end{itemize}

\end{fact}
\begin{proof}
 \n (i) $\to $ (iii). We first consider the case that $x$ is a bit  sequence $A$ in  Cantor space. Suppose $x$ is incompressible in approximation via  $b\in \NN$.

 Given $n\in \NN$, let $p$ be the special point  $A \uhr n 0^\infty$, that is, $p$ consists of the first $n$ bits of $A$, followed by   $0$s. Then, since $d(A,p) < \tp{-n}$, by our hypothesis we have $K(A\uhr n) \gep K(p,n) > n-b)$.
 This shows that $A$ is   \ML~ random.
	
	Now suppose $x\in [0,1]$ is incompressible in approximation. Then $x<1$. Let $A$ be the binary expansion of $x$ with infinitely many $0$s. If $p$ is a special point in $\cantor$ with $d(A,p) < \tp{-n}$, then the dyadic rational $q$  corresponding to $p$ is a  special point in $[0,1]$ with $d(x,q) <  \tp{-n}$. Since $q$ can be computed from $p$, this shows that $A$ is incompressible in approximation within Cantor space. Hence $x$ is \ML~ random.

	\vsp

\n  (iii) $\to $ (ii). The following argument works in both  Cantor space and  the unit interval. Let $$U_b = \bigcup_p B(p, \tp{-K(p) -b-1}).$$ Then the sequence $(U_b)\sN b $ of open sets is uniformly c.e. Furthermore, the measure of $U_b$ is bounded by $$\sum_p 2^{-K(p)-b} \leq 2^{-b}\Omega \, .$$ Thus, $\{U_b\}_{b \in \NN}$ is a \ML~ test. Since $z$ passes  this test, $z$ is incompressible in approximation.  
\end{proof}

 This only works because we have  dimension 1. In the unit \emph{square}, for instance, a ML-random $x$ will satisfy $K_*(x ; n)\gep 2n$, by a proof similar as above. In the $k$-cube .... you guessed it, $\gep kn$. In these spaces there is a  computable upper bound on $K_*(x ; n)$. This fails  for instance in  $\+C[0,1]$

We give an example of a computable metric spaces such that no incompressible point exists.
\begin{example}
	{\rm Let $\+ M$ be the Cantor space $\cantor$ with the distance function squared, i.e., with $\wt  d(f,g)= d(f,g)^2$, and the same computable structure as in Example~\ref{exam:MetricExamples}(iii). Then no point in $M$ is incompressible in approximation. }
\end{example}
To see this, suppose that $A \in M$ is incompressible in approximation via~$b$. Let $p_n$ be the special point $(A\uhr n)0^\infty$. Then $  \wt d(A, p_n)< \tp{-2n} $, whence $  K(p_n,n ) \ge K(A; 2n)> 2n - b$. For large $n$, this contradicts the fact that $K(p_n, n) \lep K(A\uhr n) \lep n+ 2 \log n$. 

\vsp

We show that  being i.a.\ is preserved under any computable map  $F \colon \+M \to \+N$  such that $F^{-1}$ is Lipschitz. 
For instance, via a computable embedding of the unit interval we can  find   points that are i.a.\ in any computable Banach space.

\begin{prop} \label{prop:inverse_Lipschitz_preserve_1} Let $F \colon \+M \to \+N$ be a computable map such that $F^{-1}$ is Lipschitz, namely, there is $v \in \NN $ such that $d(x,y) \le \tp v d(F(x), F(y))$. Then for each $z$ in $\+ M $ and each $n \ge v+1$ we have   
	
	\bc $K(F(z); n) \gep K(z; n-v-1)$. \ec 
\end{prop}

\begin{proof} Define a prefix-free machine $L$ as follows. On input $\tau$, if $\UM(\tau) = \la q, n \ra$ for a special point $q$ of $\+N$, then $L$ searches for a special point $p$ of $\+ M$ such that $d(F(p),q) < \tp{-n}$. If  $p$ is found, it outputs $\la p, n-v-1\ra$. 
	
	Suppose now that $\UM(\tau) = \la q, n \ra$ for some $\tau$ of least length,    where \bc $d(F(z), q) < \tp{-n}$. \ec Then on input $\tau$ the machine $L$ finds $p$ because $F$ is continous. This shows that $K(  p,  n-v-1  )\lep |\tau| = K(  q, n  )$. Furthermore, we have \bc $d(p,z) \le 2^v d(f(p), F(z)) \le 2^v [ d(F(p), q) + d (q, F(z))] < 2^v \tp{-n+1}$.  \ec
	Since $K(F(z); n)$ is the minimum of all such $K( q, n )$ such that $d(f(z),q) < \tp{-n}$, this establishes the required inequality. 
\end{proof}

In particular, $F$ preserves being i.a.
To preserve being strongly i.a.\ we also need that the range of $F$ is dense:

\begin{prop} \label{prop:incompr_isom} Let $F \colon \+M \to \+N$ be a computable map with range dense in $\+N$  such that $F^{-1}$ is Lipschitz with constant $2^v$. If  $z$ in $\+ M $ is  strongly i.a.\, then so is $F(z)$. In particular, being strongly i.a.\ is preserved under changing the computable structure to an equivalent one.    
\end{prop}

\begin{proof}  Suppose $z$ is strongly i.a.\ via $b$. Then, by Fact~\ref{fact:char_str_i.a.}, for each special point $p$ of $\+M$ and each $\sss$, we have 
	\begin{equation}
		\label{eqn:hyp} \UM(\sss) = p \rightarrow d(z,p ) \ge \tp{-\sssl -b}. 
	\end{equation}
	
	We define a  prefix-free machine $L$. By the recursion theorem, we may assume that a coding constant $d_L$ for $L$ is given in advance.   On input $\tau$, if $\UM(\tau) =q$ for a special point $q$ of $\+ N$,   then $L$ outputs     a special point $p$ of $\+ M$ such that \bc $d(F(p), q) <  \tp{-|\tau| -v  -b- d_L-1}$. \ec Note that $p$ exists because the range of $F$ is dense in $\+ N$.  
	
Since $d_L$ is the coding constant for $L$, we have $\UM(\sss) = p$ for some $\sss$ such that $|\sss| \le |\tau| + d_L$. Thus, by (\ref{eqn:hyp}) $d(z, p) \ge \tp{-\sssl -b} \ge \tp{-|\tau| - d_L -b}$. By the Lipschitz condition on $F^{-1}$ we obtain
	\begin{eqnarray*}
		d(F(z), q) & \ge & d(F(z),F( p)) - d(F(p), q) \\
		&\ge & 2^{-v} d(z,p) - d(F(p), q) \\
		& \ge & \tp{-|\tau| -v - d_L - b} - \tp{-|\tau| -v- d_L - b-1} \\
		& =& \tp{-|\tau| -v - d_L -b -1}. 
	\end{eqnarray*}
If $|\tau| = K(q)$ then  by Fact~\ref{fact:char_str_i.a.}, this shows that  $F(z)  $ is strongly  i.a.\   in $\+ N$  via the constant $v +d_L +b +1$.
\end{proof}

\newpage

\part{Various}

\section{Some new exercises on computability and randomness}

These will go into next edition of Nies' book. Solutions next page.

\vsp

Stephan showed that every wtt-incomplete c.e.\ set $B$ is ``i.o.\ $K$-trivial''  in the sense that $\ex^\infty n \, K(B \uhr n) \lep K(n)$. See \cite[5.2.8]{Nies:book}. The following exercise shows that the weakly $K$-trivials exist in every wtt degree.

\begin{fact} \label{exer:weakKtrivial_wtt} For each set $A$ there is a set $B \equiv_{wtt} A$ such that  \bc $\ex^\infty n \, K(B \uhr n) \lep K(n)$.  \ec 

Moreover the set of $n$ where this happens is computably bounded.    \end{fact}

We consider  the smallest cost function that makes sense at all. $\Om$ is random and hence does not obey even that. 
\begin{fact}  \label{exer:Omega2^xCostFunction} Let $c(x,s) = \tp{-x}$. Show that $\Om$ does not obey this cost function. \end{fact}

\begin{fact}  \label{exer:Omega_Arslanov} (Yu Liang) Use Arslanov's completeness criterion to give a direct proof that $\Om_0$ is low.  \end{fact}
\newpage

 Solutions: 

\n {\bf \ref{exer:weakKtrivial_wtt}}  Define sets $(B_k)_{k \ge -1}$  of size $k+1$ inductively (but not computably), as follows. Let $B_{-1} = \ES$. If $B_k$ has been defined, let $n= n_k$ be the number such that $D_{n}= B_k$ (strong index). Note that $n_k > \max B_k$.
If $k \in A$ let $B_{k+1} =  B_k \cup \{2n\}$. Otherwise, let $B_{k+1} =  B_k \cup \{2n+1\}$.

\vsp

\n {\it Verification.} Clearly $B \lwtt A$ by the recursive definition of $B$. We also  have $A \lwtt B$ because the sequence $(n_k)$ is computably bounded.

To show that $B$ is i.o.\  $K$-trivial, note that for $m = 2n_k$, we have that $K(B\uhr m) \lep K(m)$ because   $D_{n_k} = B \cap [0, m)$.

\vsp

\n {\bf \ref{exer:Omega2^xCostFunction}}   We view $\Omega_s$ as a binary string. At stage $s>0$, if there is $i$   least
such that $\Om_s(i) = 1$ and $\Om_{s-1}(i)=0$, enumerate the interval $[\Om_s\uhr{i+1}]$ into a set~$\SS$.

If $\Om$ obeys $c$ then $\SS$ is a Solovay test capturing $\Om$.

\newpage
\bibliographystyle{plain}
 
%\bibliography{../bibs/Nies,../bibs/randomness,../bibs/settheory,../bibs/various,../bibs/recursiontheory,../bibs/analysis,../bibs/kucera,../bibs/modeltheory,../bibs/reverse_maths}

\def\cprime{$'$}

\end{document}